\documentclass{amsart}

\usepackage{amsmath,amssymb,amsthm}
\usepackage{a4wide}
\usepackage{hyperref}

\usepackage{graphicx}
\usepackage{xypic}
\entrymodifiers={+!!<0pt,\fontdimen22\textfont2>}
\usepackage[all]{xy}

\newtheoremstyle{myremark} 
    {7pt}                    
    {7pt}                    
    {}  	                 
    {}                           
    {\bf}       	         
    {.}                          
    {.5em}                       
    {}  

\theoremstyle{plain}
\newtheorem{lemma}{Lemma}[section]
\newtheorem{theorem}[lemma]{Theorem}

\newtheorem{definition}[lemma]{Definition}
\newtheorem{corollary}[lemma]{Corollary}
\newtheorem{proposition}[lemma]{Proposition}

\newtheorem{notation}[lemma]{Notation}

\newtheorem*{theorem*}{Main result}

\theoremstyle{myremark}
\newtheorem{remark}[lemma]{Remark}
\newtheorem{example}[lemma]{Example}

\newcommand{\RR}{\mathbb{R}}

\newcommand{\ZZ}{\mathbb{Z}}
\newcommand{\TT}{\mathbb{T}}

\newcommand{\niceb}{\mathcal{B}}
\newcommand{\nicex}{\mathcal{X}}
\newcommand{\nicep}{\mathcal{P}}

\newcommand{\htpyequiv}{\simeq}

\newcommand{\conn}{\mathrm{conn}}
\newcommand{\redhom}{\widetilde{H}}
\newcommand{\incl}{\hookrightarrow}
\renewcommand{\subset}{\subseteq}

\newcommand{\cl}{\mathrm{Cl}}

\newcommand{\lk}{\mathrm{lk}}

\newcommand{\id}{\mathrm{id}}
\newcommand{\colim}{\mathrm{colim}}
\newcommand{\hocolim}{\mathrm{hocolim}}

\newcommand{\dirg}[1]{\overrightarrow{#1}}
\newcommand{\diredge}{\rightarrow}

\newcommand{\vrleq}[2]{\mathrm{VR}_\leq(#1;#2)}
\newcommand{\vr}[2]{\mathrm{VR}(#1;#2)}
\newcommand{\vrlessdir}[2]{\overrightarrow{\mathrm{VR}}_<(#1;#2)}
\newcommand{\vrleqdir}[2]{\overrightarrow{\mathrm{VR}}_\leq(#1;#2)}
\newcommand{\vrdir}[2]{\overrightarrow{\mathrm{VR}}(#1;#2)}
\newcommand{\vrcless}[2]{\mathbf{VR}_<(#1;#2)}
\newcommand{\vrcleq}[2]{\mathbf{VR}_\leq(#1;#2)}
\newcommand{\vrc}[2]{\mathbf{VR}(#1;#2)}
\newcommand{\wf}{\mathrm{wf}}
\newcommand{\wfc}[2]{\mathrm{wf}(#1;#2)}
\newcommand{\wfcless}[2]{\mathrm{wf}_<(#1;#2)}
\newcommand{\wfcleq}[2]{\mathrm{wf}_\leq(#1;#2)}
\newcommand{\cnk}[2]{C_{#1}^{#2}}
\newcommand{\cnkdir}[2]{\overrightarrow{C_{#1}^{#2}}}
\newcommand{\cech}{\mathbf{\check{C}}}
\newcommand{\cechless}{\mathbf{\check{C}}_{<}}
\newcommand{\cechleq}{\mathbf{\check{C}}_{\leq}}
\newcommand{\ballless}{B_{<}}
\newcommand{\ballleq}{B_{\leq}}
\newcommand{\ddist}{\vec{d}}

\newcommand{\md}[1]{\ \mathrm{mod}\ #1}

\newcommand{\expect}{\mathbf{E}}

\begin{document}
\title{The Vietoris--Rips complexes of a circle}
\author{Micha{\l} Adamaszek}
\address{Department of Mathematical Sciences, University of Copenhagen, Universitetsparken 5, 2100 Copenhagen, Denmark}
\email{aszek@mimuw.edu.pl}
\author{Henry Adams}
\address{Department of Mathematics, Colorado State University, Fort Collins, CO 80523, United States}
\email{adams@math.colostate.edu}
\thanks{Research of HA was supported by the
Institute for Mathematics and its Applications with funds provided by the National Science Foundation. MA is supported by VILLUM FONDEN through the network for Experimental Mathematics in Number Theory, Operator Algebras, and Topology. Part of this research was carried out when MA was at the Max Planck Institute for Informatics in Saarbr\"ucken.}

\begin{abstract}
Given a metric space $X$ and a distance threshold $r>0$, the Vietoris--Rips simplicial complex has as its simplices the finite subsets of $X$ of diameter less than $r$. A theorem of Jean-Claude Hausmann states that if $X$ is a Riemannian manifold and $r$ is sufficiently small, then the Vietoris--Rips complex is homotopy equivalent to the original manifold. Little is known about the behavior of Vietoris--Rips complexes for larger values of $r$, even though these complexes arise naturally in applications using persistent homology. We show that as $r$ increases, the Vietoris--Rips complex of the circle obtains the homotopy types of the circle, the 3-sphere, the 5-sphere, the 7-sphere, \ldots, until finally it is contractible. As our main tool we introduce a directed graph invariant, the \emph{winding fraction}, which in some sense is dual to the circular chromatic number. Using the winding fraction we classify the homotopy types of the Vietoris--Rips complex of an arbitrary (possibly infinite) subset of the circle, and we study the expected homotopy type of the Vietoris--Rips complex of a uniformly random sample from the circle. Moreover, we show that as the distance parameter increases, the ambient \v{C}ech complex of the circle (i.e.\ the nerve complex of the covering of a circle by all arcs of a fixed length) also obtains the homotopy types of the circle, the 3-sphere, the 5-sphere, the 7-sphere, \ldots, until finally it is contractible.
\end{abstract}
\maketitle

\section{Introduction}
\label{sect:intro}

Given a metric space $X$ and a distance threshold $r>0$, the \emph{Vietoris--Rips simplicial complex} $\vrcless{X}{r}$ has as its simplices the finite subsets of $X$ of diameter less than $r$. As the maximal simplicial complex determined by its 1-skeleton, it is an example of a clique (or flag) complex. Vietoris--Rips complexes were used by Vietoris to define a (co)homology theory for metric spaces \cite{Vietoris27}, and by Rips to study hyperbolic groups \cite{Gromov1987}.

More recently, Vietoris--Rips complexes are used in computational algebraic topology and in topological data analysis. In this context the metric space $X$ is often a finite sample from some unknown subset $M\subset \RR^d$, and one would like to use $X$ to recover topological features of $M$. The idea behind \emph{topological persistence} is to reconstruct $\vrcless{X}{r}$ as the distance threshold $r$ varies from small to large, to disregard short-lived topological features as the result of sampling noise, and to trust topological features which persist as being representative of the shape of $M$. For example, with \emph{persistent homology} one attempts to reconstruct the homology groups of $M$ from the homology of $\vrcless{X}{r}$ as $r$ varies \cite{EdelsbrunnerHarer, Carlsson2009, CarlssonIshkhanovDeSilvaZomorodian2008}.

Part of the motivation for using Vietoris--Rips complexes in applied contexts comes from the work of Hausmann and Latschev. Hausmann proves that if $M$ is a closed Riemannian manifold and if $r$ is sufficiently small compared to the injectivity radius of $M$, then $\vrcless{M}{r}$ is homotopy equivalent to $M$ \cite{Hausmann1995}. Latschev furthermore shows that if $X$ is Gromov--Hausdorff close to $M$ (for example a sufficiently dense finite sample) and $r$ is sufficiently small, then $\vrcless{X}{r}$ recovers the homotopy type of $M$ \cite{Latschev2001}. As the main idea of persistence is to allow $r$ to vary, we would like to understand what happens when $r$ is beyond the ``sufficiently small'' range.

As the main result of this paper, we show that as the distance threshold increases, the Vietoris--Rips complex $\vrcless{S^1}{r}$ of the circle obtains the homotopy types of the odd-dimensional spheres $S^1$, $S^3$, $S^5$, $S^7$, \ldots, until finally it is contractible. To our knowledge, this is the first computation for a non-contractible connected manifold $M$ of the homotopy types of $\vrcless{M}{r}$ for arbitrary $r$ (and also a first computation of the persistent homology of $\vrcless{M}{r}$). Our main result confirms, for the case $M=S^1$, a conjecture of Hausmann that for $M$ a compact Riemannian manifold, the connectivity of $\vrcless{M}{r}$ is a non-decreasing function of the distance threshold $r$ \cite[(3.12)]{Hausmann1995}.

As our main tools we introduce \emph{cyclic graphs}, a combinatorial abstraction of Vietoris--Rips graphs for subsets of the circle, and their invariant called the \emph{winding fraction}. In a sense which we make precise, the winding fraction is a directed dual of the \emph{circular chromatic number} of a graph \cite[Chapter 6]{HellNesetril2004}. In \cite{AAFPP-J} we proved that for $X\subset S^1$ finite, $\vrcless{X}{r}$ is homotopy equivalent to either an odd-dimensional sphere or a wedge sum of spheres of the same even dimension; the theory of winding fractions provides us with quantitative control over which homotopy type occurs, and also over the behavior of induced maps between complexes. As applications we classify the homotopy types of $\vrcless{X}{r}$ for arbitrary (possibly infinite) subsets $X\subset S^1$, and we analyze the evolution of the homotopy type of $\vrcless{X}{r}$ when $X\subset S^1$ is chosen uniformly at random.

\v{C}ech complexes are a second geometric construction producing a simplicial complex from a metric space. The \emph{\v{C}ech complex} $\cechless(S^1;r)$ is defined as the nerve of the collection of all open arcs of length $2r$ in the circle of circumference $1$. For $r$ small the Nerve Theorem \cite[Corollary~4G.3]{Hatcher} implies that $\cechless(S^1;r)$ is homotopy equivalent to $S^1$. Just as we study $\vrcless{S^1}{r}$ in the regime where $r$ is too large for Hausmann's result to apply, we also study $\cechless(S^1;r)$ in the regime where $r$ is too large for the Nerve Theorem to apply. We show that as $r$ increases, the ambient \v{C}ech complex $\cechless(S^1;r)$ also obtains the homotopy types of $S^1$, $S^3$, $S^5$, $S^7$, \ldots, until finally it is contractible. 

All of this has analogues for the complexes $\vrcleq{S^1}{r}$ and $\cechleq(S^1;r)$, defined by sets of diameter at most $r$, respectively closed arcs of length $2r$. We can summarize the main results as follows.

\begin{theorem*}[Theorems~\ref{thm:vrlessfull}, \ref{thm:vrcleqfull}, \ref{thm:cechleqfull}, \ref{thm:cechlessfull}]
Let $0<r<\frac12$. There are homotopy equivalences
\begin{align*}
\vrcless{S^1}{r}\htpyequiv S^{2l+1} \quad & \mathrm{if}\ \tfrac{l}{2l+1}<r\leq\tfrac{l+1}{2l+3}, \\
\cechless(S^1;r)\htpyequiv S^{2l+1} \quad & \mathrm{if}\ \tfrac{l}{2(l+1)}<r\leq\tfrac{l+1}{2(l+2)},
\end{align*}
\begin{align*}
\vrcleq{S^1}{r}\htpyequiv &\begin{cases}
S^{2l+1} & \mathrm{if}\ \frac{l}{2l+1}<r<\frac{l+1}{2l+3}, \\
\bigvee^{\mathfrak{c}} S^{2l} & \mathrm{if}\ r=\frac{l}{2l+1},
\end{cases} \\ 
\cechleq(S^1;r)\htpyequiv &\begin{cases}
S^{2l+1} & \mathrm{if}\ \frac{l}{2(l+1)}<r<\frac{l+1}{2(l+2)},\\
\bigvee^{\mathfrak{c}} S^{2l} & \mathrm{if}\ r=\frac{l}{2(l+1)},
\end{cases}
\end{align*}
where $l=0,1,\ldots$.
\end{theorem*}

\subsection*{Contents of the paper}
In Section~\ref{sect:preliminaries} we introduce preliminary concepts and notation, including Vietoris--Rips complexes. We introduce cyclic graphs and develop their invariant called the winding fraction in Section~\ref{sect:cyclic}. In Section~\ref{sect:winding} we show how this invariant affects the homotopy type of the clique complex of a cyclic graph. In Section~\ref{sect:density} we show that the homotopy type of the Vietoris--Rips complex stabilizes for sufficiently dense samples of $S^1$. We apply the winding fraction to study the evolution of Vietoris--Rips complexes for random subsets of $S^1$ in Section~\ref{sect:evolution}. The main result appears in Sections~\ref{sect:s1} and \ref{sect:singular}, where we show how to compute the homotopy types of Vietoris--Rips complexes of arbitrary (possibly infinite) subsets of $S^1$; in particular we describe $\vrcless{S^1}{r}$. In Section~\ref{sect:cech} we transfer these results to the \v{C}ech complexes of the circle. The appendices contain proofs of auxiliary results in linear algebra and probability.

\section{Preliminaries}
\label{sect:preliminaries}

We refer the reader to Hatcher \cite{Hatcher} and Kozlov \cite{Kozlov} for basic concepts in topology and combinatorial topology.

\subsection*{Simplicial complexes}
For $K$ a simplicial complex, let $V(K)$ be its vertex set. The link of a vertex $v\in V(K)$ is $\lk_K(v) = \{\sigma \in K~|~v\notin \sigma\mbox{ and }\sigma \cup \{v\} \in K\}$. We will identify an abstract simplicial complex with its geometric realization.

\begin{definition}
\label{def:cliquecomplex}
For an undirected graph $G$ the \emph{clique complex} $\cl(G)$ is the simplicial complex with vertex set $V(G)$ and with faces determined by all cliques (complete subgraphs) of $G$.
\end{definition}

\subsection*{Vietoris--Rips complexes} The Vietoris--Rips complex is used to capture a notion of proximity in a metric space. 
\begin{definition}
\label{def:vrcomplexes}
Suppose $X$ is a metric space and $r> 0$ is a real number. The \emph{Vietoris--Rips} complex $\vrcleq{X}{r}$ \emph{(}resp.\ $\vrcless{X}{r}$\emph{)} is the simplicial complex with vertex set $X$, where a finite subset $\sigma\subseteq X$ is a face if and only if the diameter of $\sigma$ is at most $r$ (resp.\ less than $r$).
\end{definition}

Every Vietoris--Rips complex is the clique complex of its $1$-skeleton. We will write $\vrc{X}{r}$, omitting the subscripts $<$ and $\leq$, in statements which remain true whenever either inequality is applied consequently throughout. 

\subsection*{Conventions regarding the circle}
We give the circle $S^1$ the arc-length metric scaled so that the circumference of $S^1$ is $1$. For $x,y\in S^1$ we denote by $[x,y]_{S^1}$ the closed clockwise arc from $x$ to $y$ and by $\ddist(x,y)$ its length --- the clockwise distance from $x$ to $y$. For a fixed choice of $0\in S^1$ each point $x\in S^1$ can be identified with the real number $\ddist(0,x)\in[0,1)$, and this will be our coordinate system on $S^1$. For any two numbers $x,y\in\RR$ we define $[x,y]_{S^1}=[x\md{1},y\md{1}]_{S^1}$ and $\ddist(x,y)=\ddist(x\md{1},y\md{1})$. Open and half-open arcs are defined similarly. If $x_1,x_2,\ldots,x_s\in S^1$ then we will write
$$x_1\prec x_2\prec\cdots\prec x_s$$
if $x_1,\ldots,x_s$ appear on $S^1$ in this clockwise order, or equivalently if they are pairwise distinct and $\sum_{i=1}^s\ddist(x_i,x_{i+1})=1$, where $x_{s+1}=x_1$. 

We define $\ddist_n(i,j)=n\cdot\ddist(\frac{i}{n},\frac{j}{n})$ to be the ``forward'' distance from $i$ to $j$ in $\ZZ/n$.

\subsection*{Directed graphs}
Throughout this work a \emph{directed graph} is a pair $\dirg{G}=(V,E)$ with $V$ the set of vertices and $E\subseteq V\times V$ the set of directed edges, subject to the conditions $(v,v)\not\in E$ (no loops) and $(v,w)\in E\implies (w,v)\not\in E$ (no edges oriented in both directions). The edge $(v,w)$ will also be denoted by $v\diredge w$. For a vertex $v\in V(\dirg{G})$ we define the out- and in-neighborhoods
$$N^+(\dirg{G},v)=\{w~:~v\diredge w\}, \quad N^-(\dirg{G},v)=\{w~:~w\diredge v\},$$
as well as their closed versions
$$N^+[\dirg{G},v]=N^+(\dirg{G},v)\cup\{v\}, \quad N^-[\dirg{G},v]=N^-(\dirg{G},v)\cup\{v\}.$$
A \emph{directed cycle} of length $s$ in $\dirg{G}$ is a sequence of vertices $v_1,\ldots,v_s$ such that there is an edge $v_i\to v_{i+1}$ for all $i=1,\ldots,s$, where $v_{s+1}=v_1$. 

For a directed graph $\dirg{G}$ we will denote by $G$ the \emph{underlying undirected graph} obtained by forgetting the orientations. If $G$ is an undirected graph we write $v\sim w$ when $v$ and $w$ are adjacent and we define
$$N(G,v)=\{w~:~w\sim v\},\quad N[G,v]=N(G,v)\cup\{v\}.$$

For $v \in V$ let $\dirg{G}\setminus v$ be the directed graph obtained by removing vertex $v$ and all of its incident edges, and for $e \in E$ let $\dirg{G}\setminus e$ be obtained by removing edge $e$. The undirected versions $G\setminus v$ and $G\setminus e$ are defined similarly.

All graphs considered in this paper are finite.

\section{Cyclic graphs, winding fractions, and dismantling}
\label{sect:cyclic}
In this section we develop the combinatorial theory of cyclic graphs, dismantling, and winding fractions.

We are going to work with the notion of a cyclic order. While there exist definitions of a cyclic order based on the abstract ternary relation of betweenness \cite{huntington1916set}, the following definition will be sufficient for our purpose. A \emph{cyclic order} on a finite set $S$ of cardinality $n$ is a bijection $h:S\to\{0,\ldots,n-1\}$. Denoting $x_i=h^{-1}(i)$ we write this simply as
$$x_0\prec x_1\prec\cdots\prec x_{n-1}.$$
If $n\geq 3$ this gives rise to a betweenness relation: we write $x_i\prec x_j\prec x_k$ if $i<j<k$ or $k<i<j$ or $j<k<i$. If $S'\subset S$ then any cyclic order on $S$ restricts in an obvious way to a cyclic order on $S'$.

A subinterval in such a cyclic ordering of $S$ is either (1) the empty set, (2) a set of the form $\{x_i,\ldots,x_j\}$ for $0\leq i\leq j\leq n-1$, or (3) a set of the form $\{x_j,\ldots,x_{n-1},x_0,\ldots,x_i\}$ for $0\leq i<j\leq n-1$. In particular $S$ itself is also a subinterval.

A function $f:S\to T$ between cyclic orders is \emph{cyclic monotone} if (1) for every $t\in T$ the set $f^{-1}(t)$ is a subinterval of $S$, and (2) $f(s)\prec f(s')\prec f(s'')$ in $T$ implies $s\prec s'\prec s''$ in $S$ for any $s,s',s''\in S$. 

One easily shows that if $f$ is cyclic monotone, if $s\prec s'\prec s''$, and if $f(s),f(s'),f(s'')$ are pairwise distinct, then $f(s)\prec f(s')\prec f(s'')$. Moreover, if $f:S\to T$ is cyclic monotone then the preimage of any subinterval of $T$ is a subinterval of $S$.


We will concentrate on the following class of directed graphs.

\begin{definition}
\label{def:cyclicgraph}
A directed graph $\dirg{G}$ is called \emph{cyclic} if its vertices can be arranged in a cyclic order $v_0\prec v_1\prec\cdots\prec v_{n-1}$ subject to the following condition: if there is a directed edge $v_i\diredge v_j$, then either $j=(i+1)\md{n}$ or there are directed edges $$v_i\diredge v_{(j-1)\md{n}} \quad\mathrm{and}\quad v_{(i+1)\md{n}}\diredge v_{j}.$$
\end{definition}
In the future all arithmetic operations on the vertex indices are understood to be reduced modulo $n$; for instance we will write simply $v_{i+k}$ for $v_{(i+k)\md{n}}$.

Two examples of cyclic graphs are shown in Fig.~\ref{fig:firstexamples}. Cyclic graphs are a special case of directed graphs with a \emph{round enumeration}; the latter are defined by the above definition when edges with double (opposite) orientations are allowed. For a comprehensive survey of related graph classes see \cite{LinSchwarcfiter2009}, especially Theorem 10.


We begin with some basic properties of cyclic graphs.
\begin{lemma}
\label{lem:cycliceasy}
Suppose $\dirg{G}$ is a cyclic graph with $n$ vertices in cyclic order $v_0\prec \cdots\prec v_{n-1}$. Then:
\begin{itemize}
\item[(a)] For every $i=0,\ldots,n-1$ there exist $s,s'\geq 0$ such that $$N^+[\dirg{G},v_i]=\{v_i,v_{i+1},\ldots,v_{i+s}\},\quad N^-[\dirg{G},v_i]=\{v_{i-s'},\ldots,v_{i-1},v_i\}.$$
\item[(b)] For every $i=0,\ldots,n-1$ we have inclusions
$$N^+(\dirg{G},v_i)\subseteq N^+[\dirg{G},v_{i+1}],\quad N^-(\dirg{G},v_{i+1})\subseteq N^-[\dirg{G},v_i].$$
\item[(c)] Every induced subgraph of $\dirg{G}$ is a cyclic graph.
\item[(d)] If $\dirg{G}$ contains a directed cycle then $v_i\to v_{i+1}$ for all $i=0,\ldots,n-1$.
\end{itemize}
\end{lemma}
\begin{proof}
Parts (a) and (b) follow directly from the definition. The cyclic orientation inherited from $\dirg{G}$ is a cyclic orientation of any induced subgraph, which gives (c). To prove (d) take a directed cycle and replace any edge $v_i\to v_j$ with $j\neq i+1$ by a path $v_i\to v_{i+1}\to v_j$. After finitely many steps of this kind we get a directed cycle in which every edge is of the form $v_i\to v_{i+1}$. 
\end{proof}

\begin{figure}
\begin{center}
\begin{tabular}{ccc}
\includegraphics[scale=0.8]{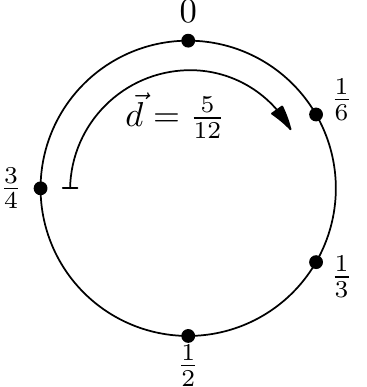} & \includegraphics[scale=0.8]{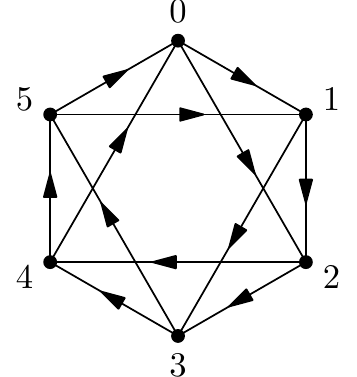} & 
\includegraphics[scale=0.8]{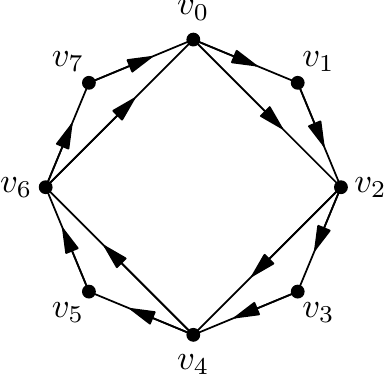}
\end{tabular}
\end{center}
\caption{\emph{(Left)} The coordinate system on $S^1$. \emph{(Middle)} The cyclic graph 
$\protect\overrightarrow{C_6^2}$. 
\emph{(Right)} A cyclic graph which is not a Vietoris--Rips graph. Its odd-numbered vertices are dominated (see Definition~\ref{def:dominated}).
}
\label{fig:firstexamples}
\end{figure}

The main examples of cyclic graphs of interest in this paper are provided in the next two definitions.

\begin{definition}
\label{def:vrdigraphs}
For a finite subset $X\subseteq S^1$ and real number $0< r<\frac12$, the directed Vietoris--Rips graphs $\vrleqdir{X}{r}$ and $\vrlessdir{X}{r}$ are defined as
$$\vrleqdir{X}{r}=\big(X,\ \{x_1\diredge x_2~:~0<\ddist(x_1,x_2)\leq r\}\big),$$
$$\vrlessdir{X}{r}=\big(X,\ \{x_1\diredge x_2~:~0<\ddist(x_1,x_2)< r\}\big).$$
\end{definition}

It is clear that the Vietoris--Rips graph is cyclic with respect to the clockwise ordering of $X$; in particular the two meanings of the symbol $\prec$ denoting clockwise order in $S^1$ and cyclic order of the vertices of $\vrdir{X}{r}$ agree.

As before, we will omit the subscript and write $\vrdir{X}{r}$ in statements which apply to both $<$ and $\leq$.  Not every cyclic graph is a Vietoris--Rips graph of a subset of $S^1$: an example is in Figure~\ref{fig:firstexamples}. Our interest in Vietoris--Rips graphs stems from the fact that a Vietoris--Rips complex is the clique complex of the corresponding undirected Vietoris--Rips graph, namely $\vrc{X}{r}=\cl(\vr{X}{r})$.

\begin{definition}
\label{def:cnkdir}
For integers $n$ and $k$ with $0\leq k < \frac12 n$, the directed graph $\cnkdir{n}{k}$ has vertex set $\{0,\ldots,n-1\}$ and edges $i\diredge {(i+s)\md{n}}$ for all $i=0,\ldots,n-1$ and $s=1,\ldots,k$. Equivalently
$$i\diredge j \quad\mbox{if and only if}\quad 0<\ddist_n(i,j)\leq k.$$
\end{definition}

The directed graphs $\cnkdir{n}{k}$ are cyclic with respect to the natural cyclic order of the vertices. Note that $\cnkdir{n}{k}$ is a Vietoris--Rips graph of the vertex set of a regular $n$-gon, or in our notation:
\begin{equation}
\label{eq:cnkvietoris}
\cnkdir{n}{k}=\vrleqdir{\{0,\tfrac{1}{n},\ldots,\tfrac{n-1}{n}\}}{\tfrac{k}{n}}.
\end{equation}
The cyclic graphs $\cnkdir{n}{k}$ will play a prominent role in our analysis of the Vietoris--Rips graphs. 

\medskip
A \emph{homomorphism of directed graphs} $f\colon\dirg{G}\to\dirg{H}$ is a vertex map such that for every edge $v\diredge w$ in $\dirg{G}$ either $f(v)=f(w)$ or there is an edge $f(v)\diredge f(w)$ in $\dirg{H}$. Directed graphs with homomorphisms form a category. We now define a class of homomorphisms for the subcategory of cyclic graphs.

\begin{definition}
\label{def:cyclichomomorphism}
Suppose $\dirg{G}$ and $\dirg{H}$ are cyclic graphs, with vertex ordering $v_0\prec\cdots\prec v_{n-1}$ in $\dirg{G}$. A vertex map $f\colon\dirg{G}\to\dirg{H}$ is a \emph{cyclic homomorphism} if $f$ is cyclic monotone, if $f$ is a homomorphism of directed graphs, and if $f$ is not constant whenever $\dirg{G}$ has a directed cycle.
\end{definition}

Note that if $\dirg{H}=\vrdir{X}{r}$ then the condition ``$f$ is cyclic monotone and not constant" is equivalent to the equation
\begin{equation}
\label{eq:degree1vr}
\sum_{i} \ddist(f(v_i),f(v_{i+1}))=1.
\end{equation}

\begin{lemma}
\label{lem:compose}
Cyclic homomorphisms have the following properties.
\begin{itemize}
\item[(a)] If $f\colon\dirg{G}\to\dirg{H}$ is a cyclic homomorphism and $\dirg{G}$ has a directed cycle then so does $\dirg{H}$.
\item[(b)] The composition of two cyclic homomorphisms is a cyclic homomorphism.
\item[(c)] The inclusion of a cyclic subgraph (with inherited cyclic orientation) is a cyclic homomorphism.
\end{itemize}
\end{lemma}
\begin{proof}
For (a), note that if $\dirg{G}$ has a directed cycle, then by Lemma~\ref{lem:cycliceasy}(d) it has a directed cycle $\cdots\diredge v_i\diredge v_{i+1}\diredge\cdots$ through all its vertices. The image of that cycle under $f$ is not constant, and by removing adjacent repetitions one gets a directed cycle in $\dirg{H}$.

For (b) suppose $f\colon\dirg{G}\to\dirg{H}$ and $g:\dirg{H}\to\dirg{K}$ are cyclic homomorphisms of cyclic graphs. We first check that $gf$ is cyclic monotone. Indeed, for a vertex $w$ in $\dirg{K}$ the set $(gf)^{-1}(w)$ is the preimage under $f$ of the subinterval $g^{-1}(w)$, hence a subinterval. If $gf(v)\prec gf(v')\prec gf(v'')$, then using that $f$ and then $g$ are cyclic monotone, we get $v\prec v'\prec v''$.

Now we only have to check that if $\dirg{G}$ has a directed cycle then $gf$ is not constant. By part (a), all three graphs have a directed cycle. Suppose, on the contrary, that $g(f(v_i))=w$ for all $v_i\in V( \dirg{G})$. Since $g$ is not constant there is a vertex $u$ of $\dirg{H}$ not in the image of $f$ with $g(u)\neq w$, and since $f$ is not constant there is an index $i$ such that $f(v_i)\prec u\prec f(v_{i+1})$ is cyclically ordered in $\dirg{H}$. Since $v_i\diredge v_{i+1}$ in $\dirg{G}$ we have $f(v_i)\diredge f(v_{i+1})$ in $\dirg{H}$. That in turn implies $f(v_i)\diredge u\diredge f(v_{i+1})$ in $\dirg{H}$ and therefore $w\to g(u)\to w$ in $\dirg{K}$. This contradicts our definition of a directed graph (no edges oriented in both directions), and hence $gf$ is not constant.

Part (c) is clear.
\end{proof}

\medskip
We can now define the main numerical invariant of cyclic graphs.
\begin{definition}
\label{def:windingfraction}
The \emph{winding fraction} of a cyclic graph $\dirg{G}$ is
\begin{equation*}
\label{eq:defwf}
\wf(\dirg{G}) = \sup\Big\{\frac{k}{n}~:~ \mathrm{there\ exists\ a\ cyclic\ homomorphism}\ \cnkdir{n}{k}\to\dirg{G} \Big\}.
\end{equation*}
\end{definition}
For a finite subset $X\subseteq S^1$ we also introduce a shorthand notation
$$\wfcleq{X}{r}=\wf(\vrleqdir{X}{r}), \ \wfcless{X}{r}=\wf(\vrlessdir{X}{r}).$$

The next proposition records the basic properties of the winding fraction.

\begin{proposition}
\label{prop:wfbasic}
The winding fraction satisfies the following properties.
\begin{itemize}
\item[(a)] $\wf(\dirg{G})>0$ if and only if $\dirg{G}$ has a directed cycle.
\item[(b)] If $\dirg{G}\to\dirg{H}$ is a cyclic homomorphism then $\wf(\dirg{G})\leq\wf(\dirg{H})$.
\item[(c)] If $X\subseteq S^1$ is a finite set and $0< r<\frac12$ then $$\wfcleq{X}{r}\leq r,\quad \wfcless{X}{r}<r,\quad \wfcless{X}{r}\leq\wfcleq{X}{r}.$$
\item[(d)] $\wf(\dirg{\cnk{n}{k}})=\frac{k}{n}$.
\end{itemize}
\end{proposition}
\begin{proof} 
For (a) note that if $\dirg{G}$ has a directed cycle, then by Lemma~\ref{lem:cycliceasy}(d) the map $i\mapsto v_i$ defines a cyclic homomorphism $\cnkdir{n}{1}\to\dirg{G}$ with $n=|V(\dirg{G})|$. Conversely, if $\dirg{G}$ has no directed cycle then by Lemma~\ref{lem:compose}(a) it admits cyclic homomorphisms only from the graphs $\cnkdir{n}{0}$.

Part (b) follows from the definition of the winding fraction and the fact that a composition of cyclic homomorphisms is a cyclic homomorphism. 

Now we prove the first inequality of (c). Suppose that $f\colon\cnkdir{n}{k}\to\vrleqdir{X}{r}$ is a cyclic homomorphism with $k\geq 1$, which means that for every $i=0,\ldots,n-1$ we have $\ddist(f(i),f(i+k))\leq r$. Since every arc of the form $[f(j),f(j+1)]_{S^1}$ is covered by exactly $k$ arcs $[f(i),f(i+k)]_{S^1}$, we have
$$nr\geq \sum_i \ddist(f(i),f(i+k))=k\sum_j \ddist(f(j),f(j+1))=k,$$
where in the last equality we used \eqref{eq:degree1vr}. It follows that $\frac{k}{n}\leq r$ and $\wfcleq{X}{r}\leq r$.

The second inequality has similar proof with strict inequalities and the third one follows from (b) since we have a subgraph inclusion $\vrlessdir{X}{r}\incl \vrleqdir{X}{r}$.

For (d), the identity automorphism of $\cnkdir{n}{k}$ shows $\wf(\cnkdir{n}{k})\geq \frac{k}{n}$. Conversely, applying part (c) with $X=\{0,\frac{1}{n},\ldots,\frac{n-1}{n}\}$ and $r=\frac{k}{n}$ gives, by \eqref{eq:cnkvietoris}, that $\wf(\cnkdir{n}{k})\leq \frac{k}{n}$.
\end{proof}

\medskip
We now describe a practical way of computing the winding fraction. The method uses graph reductions modelled on the notion of \emph{dismantlings} of undirected graphs (called \emph{folds} in \cite[Section 2.11]{HellNesetril2004} or \emph{LC reductions} in \cite{Matouvsek2008}), and hence we use the same terminology.

\begin{definition}
\label{def:dominated}
Suppose $\dirg{G}$ is a cyclic graph with vertex ordering $v_0\prec\cdots\prec v_{n-1}$. A vertex $v_i$ is called \emph{dominated by} $v_{i+1}$ (or just \emph{dominated}) if $N^-(\dirg{G},v_{i+1})=N^-[\dirg{G},v_i]$.
\end{definition}

\begin{lemma}
\label{lem:removedominated}
If $\dirg{G}$ is a cyclic graph and $v_i$ is dominated by $v_{i+1}$, then the map $f\colon\dirg{G}\to\dirg{G}\setminus v_i$ given by
\begin{equation*}
f(v_j)=\begin{cases}v_j &\mathrm{if}\ j\neq i\\ v_{i+1}&\mathrm{if}\ j=i\end{cases}
\end{equation*}
is a cyclic homomorphism. The composition $\dirg{G}\setminus v_i\incl \dirg{G}\xrightarrow{f}\dirg{G}\setminus v_i$ is the identity.
\end{lemma}
\begin{proof}
We first check that $f$ is a homomorphism of directed graphs. First note the map $f$ preserves all edges avoiding $v_i$. If $v_k\to v_i$ then $v_k\to v_{i+1}$ because $v_i$ is dominated by $v_{i+1}$. If $v_i\to v_k$ then either $k=i+1$, and then $f(v_i)=f(v_k)$, or there is also an edge $v_{i+1}\to v_k$ because $\dirg{G}$ is cyclic.

The map $f$ is a cyclic homomorphism because it clearly preserves the cyclic ordering, and if $\dirg{G}$ has a directed cycle then it has at least three vertices, in which case $f$ is not constant.

The last claim is obvious.
\end{proof}

The removal of a dominated vertex can be repeated as long as the new graph has a dominated vertex.

\begin{definition}
\label{def:dismantling}
We say a cyclic graph $\dirg{G}$ \emph{dismantles} to an induced subgraph $\dirg{H}$ if there is a sequence of graphs $\dirg{G}=\dirg{G}_0,\dirg{G}_1,\ldots,\dirg{G}_s=\dirg{H}$ such that $\dirg{G}_i$ is obtained from $\dirg{G}_{i-1}$ by removing a dominated vertex for $i=1,\ldots,s$. 
\end{definition}

If $\dirg{G}$ dismantles to $\dirg{H}$ then the composition of cyclic homomorphisms $\dirg{G}_i\to\dirg{G}_{i+1}$ provided by Lemma~\ref{lem:removedominated} gives a cyclic homomorphism $\dirg{G}\to\dirg{H}$. Moreover the composition
$$\dirg{H}\incl \dirg{G}\to\dirg{H}$$
is the identity of $\dirg{H}$. The next proposition answers the question of when the dismantling process of a cyclic graph must stop.

\begin{proposition}
\label{prop:cnkisfinal}
A cyclic graph without a dominated vertex is isomorphic to $\cnkdir{n}{k}$ for some $0\leq k<\frac12 n$. As a consequence every cyclic graph dismantles to an induced subgraph of the form $\cnkdir{n}{k}$.
\end{proposition}
\begin{proof}
Let $\dirg{G}$ be a cyclic graph with vertex ordering $v_0\prec\cdots\prec v_{n-1}$ and with no dominated vertex. By Lemma~\ref{lem:cycliceasy}(a) for every $j=0,\ldots,n-1$ there is an $e(j)$ such that $N^+[\dirg{G},v_j]=\{v_{j},\ldots,v_{e(j)}\}$. For every $i=0,\ldots,n-1$ we have
$$N^-[\dirg{G},v_i]\setminus N^-(\dirg{G},v_{i+1})\ =\ \{v_j~:~e(j)=i\},$$
where $N^-(\dirg{G},v_{i+1})\subseteq N^-[\dirg{G},v_i]$ by Lemma~\ref{lem:cycliceasy}(b). It follows that
$$\sum_i|N^-[\dirg{G},v_i]\setminus N^-(\dirg{G},v_{i+1})|=n.$$
Since $\dirg{G}$ has no dominated vertices, all $n$ summands above are positive and therefore all are equal to $1$. We have
$$|N^-(\dirg{G},v_{i+1})|=|N^-[\dirg{G},v_i]|-1=|N^-(\dirg{G},v_i)\cup\{v_i\}|-1=|N^-(\dirg{G},v_i)|.$$
Denote the common value of $|N^-(\dirg{G},v_i)|$ by $k$. Using Lemma~\ref{lem:cycliceasy}(a) again we see that $N^-[\dirg{G},v_i]=\{v_{i-k},\ldots,v_i\}$ for all $i$, and so $\dirg{G}$ is isomorphic to $\cnkdir{n}{k}$.
\end{proof}

\begin{remark}
\label{rem:uniquecore}
In \cite{AAM} we prove that, regardless of the choices of a dominated vertex made in the process, every dismantling of a cyclic graph ends up with the same subgraph. Such strong uniqueness is not needed in this paper. 

Our notion is modelled on the more classical dismantling of undirected graphs (see \cite{HellNesetril2004}). In that setting the end result of the dismantling process is unique only up to isomorphism (see \cite{Matouvsek2008} or \cite[Theorem 2.60]{HellNesetril2004}).
\end{remark}

We can now give a recipe for computing the winding fraction.

\begin{proposition}
\label{prop:wfofcore}
If a cyclic graph $\dirg{G}$ dismantles to $\cnkdir{n}{k}$ then $\wf(\dirg{G})=\frac{k}{n}$.
\end{proposition}
\begin{proof}
The graph $\dirg{G}$ has both cyclic homomorphisms $\cnkdir{n}{k}\incl \dirg{G}$ and $\dirg{G}\to\cnkdir{n}{k}$, so the claim follows from Proposition~\ref{prop:wfbasic} parts (b) and (d).
\end{proof}

The following result gives the converse of Proposition~\ref{prop:wfbasic}(b).

\begin{proposition}
\label{prop:cychomoall}
There is a cyclic homomorphism $f\colon\dirg{G}\to\dirg{H}$ if and only if $\wf(\dirg{G})\leq\wf(\dirg{H})$.
\end{proposition}
\begin{proof}
The ``only if'' part is handled by Proposition~\ref{prop:wfbasic}(b). 

For any $0\leq k<\frac12 n$ and $d\geq 1$ consider two maps $\iota: \cnkdir{n}{k}\to\cnkdir{nd}{kd}$ and $\tau:\cnkdir{nd}{kd}\to\cnkdir{n}{k}$ given by
$$\iota(i)=di, \quad \tau(j)=\lfloor \tfrac{j}{d}\rfloor.$$
It easy to see that $\iota$ and $\tau$ are cyclic homomorphisms. 

To prove the ``if'' part, suppose that $\dirg{G}$ dismantles to $\cnkdir{n}{k}$ and $\dirg{H}$ dismantles to $\cnkdir{n'}{k'}$. Proposition~\ref{prop:wfofcore} and the assumption $\wf(\dirg{G})\leq\wf(\dirg{H})$ imply $\frac{k}{n}\leq\frac{k'}{n'}$. Then we have a cyclic homomorphism
$$\dirg{G}\to \cnkdir{n}{k}\xrightarrow{\iota}\cnkdir{nn'}{kn'}\incl\cnkdir{nn'}{nk'}\xrightarrow{\tau}\cnkdir{n'}{k'}\incl\dirg{H}$$
where the first and last map come from dismantling, and the middle map is a subgraph inclusion since $kn'\leq nk'$.
\end{proof}

The winding fraction is in a sense dual to the well-studied concept of \emph{circular chromatic number}, see \cite[Chapter 6]{HellNesetril2004}. For an arbitrary undirected graph $G$ the circular chromatic number $\chi_c(G)$ is defined as the infimum over numbers $\frac{n}{k}$ such that there is a map $V(G)\to\ZZ/n$ which maps every edge to a pair of numbers \emph{at least} $k$ apart. By Proposition~\ref{prop:cychomoall} we have
$$\wf(\dirg{G}) = \inf\Big\{\frac{k}{n}~:~ \mathrm{there\ exists\ a\ cyclic\ homomorphism}\ \dirg{G}\to\cnkdir{n}{k} \Big\}$$
which leads to the following description: $\wf(\dirg{G})$ is the infimum over numbers $\frac{k}{n}$ such that there is an order-preserving map $V(G)\to\ZZ/n$ which maps every edge to a pair of numbers \emph{at most} $k$ apart.

\section{Winding fractions determine homotopy types}
\label{sect:winding}

We now analyze the influence of the winding fraction $\wf(\dirg{G})$ on the topology of the clique complex $\cl(G)$.

A homomorphism $f\colon G\to H$ of undirected graphs is a vertex map such that $v\sim w$ implies $f(v)=f(w)$ or $f(v)\sim f(w)$. Every homomorphism of directed graphs $\dirg{G}\to\dirg{H}$ determines a homomorphism of the underlying undirected graphs $G\to H$, and in turn also a simplicial map $\cl(G)\to\cl(H)$. The assignment $\dirg{G}\mapsto\cl(G)$ is a functor from the category of directed graphs to topological spaces, and also a functor from the subcategory of cyclic graphs to topological spaces.

\begin{lemma}
\label{lem:removedominatedclique}
If $\dirg{G}$ is a cyclic graph and $v_i$ is a dominated vertex, then the cyclic homomorphisms $\dirg{G}\setminus v_i\incl \dirg{G}$ and $\dirg{G}\to\dirg{G}\setminus v_i$ from Lemma~\ref{lem:removedominated} induce homotopy equivalences of clique complexes. 
\end{lemma}
\begin{proof}
Using the conditions listed in Lemma~\ref{lem:cycliceasy}(b) and Definition~\ref{def:dominated} we get
\begin{equation*}
N[G,v_i]=N^-[\dirg{G},v_i]\cup N^+(\dirg{G},v_i)
\subseteq N^-(\dirg{G},v_{i+1})\cup N^+[\dirg{G},v_{i+1}]=N[G,v_{i+1}].
\end{equation*}
Hence the link $\lk_{\cl(G)}(v_i)$ is a cone with apex $v_{i+1}$, or in other words, $\cl(G)$ is obtained from $\cl(G\setminus v_i)$ by attaching a cone over a cone. It follows that the inclusion $\cl(G\setminus v_i)\incl\cl(G)$ is a homotopy equivalence. Since the composition $\cl(G\setminus v_i)\incl \cl(G)\to \cl(G\setminus v_i)$ is the identity, also $\dirg{G}\to\dirg{G}\setminus v_i$ induces a homotopy equivalence.
\end{proof}

\begin{corollary}
\label{cor:dismantlinghtpy}
If a cyclic graph $\dirg{G}$ dismantles to $\dirg{H}$ then the maps of clique complexes induced by $\dirg{H}\incl\dirg{G}$ and $\dirg{G}\to\dirg{H}$ are homotopy equivalences.
\end{corollary}

To determine the homotopy types of $\cl(G)$ for arbitrary cyclic graphs $\dirg{G}$ we recall the following result, proved with different methods in \cite{Adamaszek2013} and \cite{AAFPP-J}.

\begin{theorem}
\label{thm:cnk}
For $0\leq k<\frac12 n$ there are homotopy equivalences
\begin{equation*}
\cl(\cnk{n}{k})\htpyequiv\begin{cases}
S^{2l+1} & \mathrm{if}\ \frac{l}{2l+1}<\frac{k}{n}<\frac{l+1}{2l+3}\ \mathrm{for\ some}\ l=0,1,\ldots,\\
\bigvee^{n-2k-1}S^{2l} & \mathrm{if}\ \frac{k}{n}=\frac{l}{2l+1}\ \mathrm{for\ some}\ l=0,1,\ldots.
\end{cases}
\end{equation*}
\end{theorem}
By convention an empty wedge sum is a point. We immediately obtain the following result.
\begin{theorem}
\label{thm:cyclichtpy}
If $\dirg{G}$ is a cyclic graph then
\begin{equation*}
\cl(G)\htpyequiv\begin{cases}
S^{2l+1} & \mathrm{if}\ \frac{l}{2l+1}<\wf(\dirg{G})<\frac{l+1}{2l+3}\ \mathrm{for\ some}\ l=0,1,\ldots,\\
\bigvee^{n-2k-1}S^{2l} & \mathrm{if}\ \wf(\dirg{G})=\frac{l}{2l+1}\ \mathrm{and}\ \dirg{G}\ \mathrm{dismantles\ to}\ \cnkdir{n}{k}.
\end{cases}
\end{equation*}
\end{theorem}
\begin{proof}
Graph $\dirg{G}$ dismantles to some $\cnkdir{n}{k}$ for $0\leq k<\frac12 n$ by Proposition~\ref{prop:cnkisfinal}, and then we have $\cl(G)\htpyequiv\cl(\cnk{n}{k})$ by Corollary~\ref{cor:dismantlinghtpy}. From  Proposition~\ref{prop:wfofcore} we get $\wf(\dirg{G})=\frac{k}{n}$, and plugging this into Theorem~\ref{thm:cnk} gives the result.
\end{proof}

\begin{corollary}
\label{cor:vrchtpy}
If $X\subseteq S^1$ is a finite set and $0\leq r<\frac12$ then
\begin{equation*}
\vrc{X}{r}\htpyequiv\begin{cases}
S^{2l+1} & \mathrm{if}\ \frac{l}{2l+1}<\wfc{X}{r}<\frac{l+1}{2l+3}\ \mathrm{for\ some}\ l=0,1,\ldots,\\
\bigvee^{n-2k-1}S^{2l} & \mathrm{if}\ \wfc{X}{r}=\frac{l}{2l+1}\ \mathrm{and}\ \vrdir{X}{r}\ \mathrm{dismantles\ to}\ \cnkdir{n}{k}.
\end{cases}
\end{equation*}
\end{corollary}
\begin{proof}
For the cyclic graph $\vrdir{X}{r}$ we have $\vrc{X}{r}=\cl(\vr{X}{r})$.
\end{proof}

\begin{remark}
A \emph{circular-arc graph} (CA) is an intersection graph of a collection of arcs in $S^1$. A circular-arc graph is \emph{proper} (PCA) if no arc contains another and \emph{unit} (UCA) if all arcs have the same length. We have inclusions of graph classes $\mathrm{UCA}\subsetneq \mathrm{PCA} \subsetneq \mathrm{CA}$. If $\dirg{G}$ is a cyclic graph then one can show $G$ is a PCA graph, and if $X\subset S^1$ is finite and $0\leq r<\frac12$ then the Vietoris--Rips graph $\vr{X}{r}$ is a UCA graph. In \cite{AAFPP-J} we proved that the clique complex of any CA graph has the homotopy type of $S^{2l+1}$ or a wedge of copies of $S^{2l}$ for some $l\geq 0$. The theory of winding fractions refines the result of \cite{AAFPP-J} by providing quantitative control over which homotopy type occurs, and by allowing us to understand induced maps. These features will be crucial for the applications we present in the following sections.
\end{remark}

There is a clear difference in the behaviour of $\cl(G)$ when $\wf(\dirg{G})$ is one of the \emph{singular} values $\frac{l}{2l+1}$, $l=0,1,\ldots$ as opposed to a \emph{generic} value $\frac{l}{2l+1}<\wf(\dirg{G})<\frac{l+1}{2l+3}$. We now discuss additional properties of $\cl(G)$ in the generic situation. The next lemmas describe the effect of a vertex or edge removal on the homotopy type of $\cl(G)$.

\begin{lemma}
\label{lem:removeone}
Suppose that $\dirg{G}$ is a cyclic graph and $v\in V(\dirg{G})$. If
$\frac{l}{2l+1}<\wf(\dirg{G}\setminus v)\leq\wf(\dirg{G})<\frac{l+1}{2l+3}$,
then the inclusion $\dirg{G}\setminus v\incl\dirg{G}$
induces a homotopy equivalence of clique complexes.
\end{lemma}
\begin{proof}
By Theorem~\ref{thm:cyclichtpy} the complexes $\cl(G\setminus v)$ and $\cl(G)$ are both homotopy equivalent to $S^{2l+1}$. Let $\dirg{G}_v$ denote the cyclic subgraph of $\dirg{G}$ induced by $N(G,v)$, so that  $\lk_{\cl(G)}(v)=\cl(G_v)$. The decomposition $\cl(G)=\cl(G\setminus v)\bigcup_{\cl(G_v)} (\cl(G_v)\ast v)$ yields a Mayer--Vietoris long exact sequence of homology groups whose only nontrivial part is
\begin{equation}
\label{eq:lesvertex}
\xymatrix@R0pt@C12pt{
0\ar[r] & \redhom_{2l+1}(\cl(G_v))\ar[r] & \redhom_{2l+1}(\cl(G\setminus v))\ar[r] & \redhom_{2l+1}(\cl(G))\ar[r] & \redhom_{2l}(\cl(G_v))\ar[r] & 0 \\
 & & \| & \| & & \\
 & & \ZZ & \ZZ & &
}.
\end{equation}
Since $\dirg{G}_v$ is cyclic, by Theorem~\ref{thm:cyclichtpy} the homology $\redhom_*(\cl(G_v))$ is free and concentrated in at most one dimension. In view of \eqref{eq:lesvertex} this is possible only if $\redhom_*(\cl(G_v))=0$ and the middle map in \eqref{eq:lesvertex} is an isomorphism. So $\dirg{G}\setminus v\incl\dirg{G}$ induces a homology isomorphism between spaces homotopy equivalent to $S^{2l+1}$, and hence is a homotopy equivalence by the Hurewicz and Whitehead theorems.
\end{proof}

\begin{lemma}
\label{lem:removeoneedge}
Suppose that $\dirg{G}$ is a cyclic graph and $e\in E(\dirg{G})$ is an edge such that $\dirg{G}\setminus e$ is also a cyclic graph. If
$\frac{l}{2l+1}<\wf(\dirg{G}\setminus e)\leq\wf(\dirg{G})<\frac{l+1}{2l+3}$,
then the inclusion $\dirg{G}\setminus e\incl\dirg{G}$
induces a homotopy equivalence of clique complexes.
\end{lemma}
\begin{proof}
Let $e=(a,b)$ and denote by $\dirg{G}_e$ the cyclic subgraph of $\dirg{G}$ induced by $N(G,a)\cap N(G,b)$. Then we have a decomposition $$\cl(G)=\cl(G\setminus e)\bigcup_{\cl(G_e)\ast\{a,b\}} (\cl(G_e)\ast e)=\cl(G\setminus e)\bigcup_{\Sigma\ \cl(G_e)} (\cl(G_e)\ast e).$$
By Mayer--Vietoris this yields the exact sequence
\begin{equation*}
\xymatrix@R0pt@C12pt{
0\ar[r] & \redhom_{2l}(\cl(G_e))\ar[r] & \redhom_{2l+1}(\cl(G\setminus e))\ar[r] & \redhom_{2l+1}(\cl(G))\ar[r] & \redhom_{2l-1}(\cl(G_e))\ar[r] & 0 \\
 & & \| & \| & & \\
 & & \ZZ & \ZZ & &
},
\end{equation*}
where $\redhom_{k}(\cl(G_e))=\redhom_{k+1}(\Sigma\ \cl(G_e))$. The proof can now be completed as in Lemma~\ref{lem:removeone}.
\end{proof}

\begin{proposition}
\label{prop:inducedhtpyequiv}
Suppose $f\colon\dirg{G}\to\dirg{H}$ is a cyclic homomorphism and $\frac{l}{2l+1}<\wf(\dirg{G})\leq \wf(\dirg{H})<\frac{l+1}{2l+3}$. Then $f$ induces a homotopy equivalence of clique complexes.
\end{proposition}
\begin{proof}
We proceed in three stages. First, suppose that $f\colon\dirg{G}\to\dirg{H}$ is injective on the vertices, i.e.\ it is an inclusion of a subgraph (not necessarily induced). In that case $f$ can be factored as a composition of cyclic homomorphisms
$$\dirg{G}=\dirg{G}_0\incl\dirg{G}_1\incl\cdots\incl\dirg{G}_s=\dirg{H}$$
where each inclusion $\dirg{G_i}\incl\dirg{G_{i+1}}$ is an extension by a single vertex or by a single edge. Since $\frac{l}{2l+1}<\wf(\dirg{G})\leq\wf(\dirg{G}_i)\leq\wf(\dirg{H})<\frac{l+1}{2l+3}$ the result follows from Lemmas~\ref{lem:removeone} and \ref{lem:removeoneedge}.

Next, we prove the statement for an arbitrary cyclic homomorphism $f\colon\cnkdir{n}{k}\to\cnkdir{n'}{k'}$ with $\frac{l}{2l+1}<\frac{k}{n}\leq\frac{k'}{n'}<\frac{l+1}{2l+3}$. Our first goal is to find a factorization $f=\tau\circ f_d$ where $f_d\colon\cnkdir{n}{k}\to\cnkdir{dn'}{dk'}$ is injective and $\tau\colon\cnkdir{dn'}{dk'}\to\cnkdir{n'}{k'}$ is given by $\tau(j)=\lfloor \tfrac{j}{d}\rfloor$.

Let $j_0\prec \cdots\prec j_s$, with $1\leq s\leq n-1$, be the cyclically ordered vertices of the image of $f$ in $\cnkdir{n'}{k'}$. Since $f$ is a cyclic homomorphism, each preimage $f^{-1}(j_q)$ is an interval modulo $n$. Define the cyclically ordered vertices $i_0\prec\cdots\prec i_s$ in $\cnkdir{n}{k}$ by $f^{-1}(j_q)=\{i_q,\ldots,i_{q+1}-1\}$. Choose $d\geq \max\{|f^{-1}(j_q)|, q=0,\ldots,s\}$ and define a map $f_d:\cnkdir{n}{k}\to\cnkdir{dn'}{dk'}$ by
$$f_d(i)=dj_q+\ddist_n(i_q,i) \quad\mathrm{for}\ i\in\{i_q,\ldots,i_{q+1}-1\}.$$
Note that $0\leq i-i_q<|f^{-1}(j_q)|\leq d$; therefore $f_d$ preserves the cyclic ordering and hence is a cyclic homomorphism so long as it is a homomorphism of directed graphs. It suffices to check that for every $i=0,\ldots,n-1$ we have
$$\ddist_{dn'}(f_d(i),f_d(i+k))\leq dk'.$$
Suppose $i\in f^{-1}(j_q)$ and $i+k\in f^{-1}(j_{q'})$; necessarily $\ddist_{n'}(j_q,j_{q'})\leq k'$. If $\ddist_{n'}(j_q,j_{q'})\leq k'-1$ then
$$\ddist_{dn'}(f_d(i),f_d(i+k))\leq \ddist_{dn'}(dj_q,dj_{q'}+d)\leq dk'. $$
If $j_{q'}=j_q+k'$ then
$$\ddist_{dn'}(f_d(i),f_d(i+k))=dk'+\ddist_n(i_{q'},i+k)-\ddist_n(i_q,i)=dk'+\ddist_n(i,i+k)-\ddist_n(i_q,i_{q'})=dk'+k-\ddist_n(i_q,i_{q'}).$$
We have $\ddist_n(i_q,i_{q'})\geq k$, for otherwise $\ddist_n(i_q-1,i_{q'})\leq k$ and $\ddist_{n'}(f(i_q-1),f(i_{q'}))=\ddist_{n'}(j_{q-1},j_{q'})\geq k'+1$ would contradict the fact that $f$ is a homomorphism. This ends the proof that $f_d$ is a cyclic homomorphism.

Consider the two cyclic homomorphisms $\iota: \cnkdir{n'}{k'}\to\cnkdir{dn'}{dk'}$ and $\tau:\cnkdir{dn'}{dk'}\to\cnkdir{n'}{k'}$ given by
$$\iota(i)=di, \quad \tau(j)=\lfloor \tfrac{j}{d}\rfloor.$$
 We have a commutative diagram
\begin{equation*}
\xymatrix@C50pt{
\cnkdir{n}{k}  \ar[rd]_f \ar[r]^{f_d}_\htpyequiv & \cnkdir{dn'}{dk'} \ar[d]^\tau & \cnkdir{n'}{k'} \ar[l]_\iota^\htpyequiv  \ar[ld]_\htpyequiv^\id\\
&  \cnkdir{n'}{k'} &
}
\end{equation*}
where $\htpyequiv$ indicates the map induces a homotopy equivalence of clique complexes; for the inclusions $f_d$ and $\iota$ this follows from the first part of the proof. From the diagram we conclude that $f$ induces a homotopy equivalence.

Finally, to prove the general case, suppose that $\dirg{G}$ dismantles to $\cnkdir{n}{k}$ and $\dirg{H}$ dismantles to $\cnkdir{n'}{k'}$ with $\wf(\dirg{G})=\frac{k}{n}\le\frac{k'}{n'}=\wf(\dirg{H})$. The composition
$$
\xymatrix@C20pt{
\cnkdir{n}{k} \ar@{^{(}->}[r]^\htpyequiv & \dirg{G} \ar[r]^f & \dirg{H} \ar[r]^\htpyequiv & \cnkdir{n'}{k'}
}
$$
induces a homotopy equivalence of clique complexes, and therefore so does $f$.
\end{proof}

We defer until Section~\ref{sect:singular} a further study of the combinatorics of $\cl(G)$ when $\wf(\dirg{G})=\frac{l}{2l+1}$ is a singular value.

\section{Density implies stability}
\label{sect:density}

In this section we make precise the heuristic observation that the winding fraction $\wfc{X}{r}$ increases with the density of $X$ in $S^1$. For this we recall the notion of covering from metric geometry.

\begin{definition}
\label{def:epsiloncovering}
A subset $X$ of a metric space $M$ is an \emph{$\varepsilon$-covering} if every point of $M$ is within distance less than $\varepsilon$ from some point in $X$.
\end{definition}

A finite subset $X\subset S^1$ is an $\varepsilon$-covering of $S^1$ if and only if every two cyclically consecutive points in $X$ are less than $2\varepsilon$ apart.

As motivation for this section, we note that if $0 <r<\frac13$ and $X\subset S^1$ is a finite subset, then $\vrcless{X}{r}\htpyequiv S^1$ if and only if $X$ is an $(r/2)$-covering of $S^1$. The next proposition is an analogue of this observation for bigger winding fractions and therefore for higher-dimensional homotopy types of $\vrcless{X}{r}$.


\begin{proposition}
\label{prop:approachinghighwf}
Suppose that $0<r<\frac12$ and $X\subseteq S^1$ is a finite subset. If $X$ is an $\varepsilon$-covering of $S^1$ for some $\varepsilon>0$ then $\wfcless{X}{r}>r-2\varepsilon$.
\end{proposition}
\begin{proof}
We can assume that $r-2\varepsilon>0$. There exists an $\varepsilon'<\varepsilon$ such that $X$ is also an $\varepsilon'$-covering. It suffices to show that whenever $0<\frac{k}{n}< r-2\varepsilon'$ then there is a cyclic homomorphism $\cnkdir{n}{k}\to\vrlessdir{X}{r}$, since then we get
$$\wfcless{X}{r}\geq r-2\varepsilon'>r-2\varepsilon.$$

Fix $0<\frac{k}{n}< r-2\varepsilon'$. For every $i=0,\ldots,n-1$ let $x_i\in X$ be the point closest to $\frac{i}{n}$. (The uniqueness of each $x_i$ can be assured by an infinitesimal rotation, if necessary) Then $x_0,\ldots x_{n-1}$ appear on $S^1$ in this clockwise order (possibly with repetitions) and, since $\varepsilon'<\frac12 r<\frac14$, not all of the $x_i$ are the same. By the triangle inequality
\begin{eqnarray*}
d(x_i,x_{i+k})&\leq & d(x_i,\tfrac{i}{n})+d(\tfrac{i}{n},\tfrac{i+k}{n})+d(\tfrac{i+k}{n},x_{i+k})\\
&< & \varepsilon'+\tfrac{k}{n}+\varepsilon'=\tfrac{k}{n}+2\varepsilon'< r.
\end{eqnarray*}
It follows that the map $i\mapsto x_i$ determines a cyclic homomorphism $\cnkdir{n}{k}\to\vrlessdir{X}{r}$, and the proof is complete.
\end{proof}

This leads to the following conclusion.

\begin{proposition}
\label{prop:denseinduced}
Suppose that $\frac{l}{2l+1}<r\leq r'<\frac{l+1}{2l+3}$ and $\delta=r-\frac{l}{2l+1}$. If $X\subset Y$ are finite subsets of $S^1$ and $X$ is a $(\frac12\delta)$-covering of $S^1$, then in the diagram
$$
\xymatrix{
\vrcleq{X}{r} \ar@{^{(}->}[r] & \vrcleq{Y}{r'}\\
\vrcless{X}{r} \ar@{^{(}->}[r] \ar@{^{(}->}[u] & \vrcless{Y}{r'} \ar@{^{(}->}[u]\\
}
$$
all spaces are homotopy equivalent to $S^{2l+1}$ and all maps are homotopy equivalences. 

For the spaces in the bottom row and the bottom map the same conclusion holds under the weaker assumption $\frac{l}{2l+1}<r\leq r'\leq \frac{l+1}{2l+3}$.
\end{proposition}
\begin{proof}
Proposition~\ref{prop:approachinghighwf} gives $\wfcless{X}{r}>r-\delta=\frac{l}{2l+1}$ and by Lemma~\ref{lem:cycliceasy}(b) we have $\wfcleq{Y}{r'}\leq r'<\frac{l+1}{2l+3}$. Hence the four cyclic graphs underlying the diagram have their winding fractions in the open interval $(\frac{l}{2l+1},\frac{l+1}{2l+3})$. The statement now follows from Proposition~\ref{prop:inducedhtpyequiv}.

If $r'=\frac{l+1}{2l+3}$ then by Proposition~\ref{prop:wfbasic}(c) we still have $\wfcless{Y}{r'}<r'=\frac{l+1}{2l+3}$ and Proposition~\ref{prop:inducedhtpyequiv} applies in the bottom row.
\end{proof}

We end this section with a partial converse of Proposition~\ref{prop:approachinghighwf}.

\begin{proposition}
\label{prop:highwfisdense}
Suppose that $\frac{l}{2l+1}<r$ and $\delta=r-\frac{l}{2l+1}$. If $X\subseteq S^1$ is a finite subset with  $\wfcless{X}{r}>\frac{l}{2l+1}$ then $X$ is a  $\left(\left(l+\frac12\right)\delta\right)$-covering of $S^1$.
\end{proposition}
\begin{proof}
Suppose that $\vrlessdir{X}{r}$ dismantles to $\cnkdir{n}{k}$ with $\frac{k}{n}>\frac{l}{2l+1}$. Let $x_0\prec\cdots\prec x_{n-1}$ be the points of $X$ which induce the subgraph $\cnkdir{n}{k}\incl \vrlessdir{X}{r}$. The proof will be complete if we show the following claim: for every $i$ there exists a $j\neq i$ such that $\ddist(x_i,x_j)<(2l+1)\delta$. Without loss of generality it suffices to prove this for $i=0$. We can assume that $(2l+1)\delta<1$, for otherwise the claim is trivial.

Consider the directed path in $\vrlessdir{X}{r}$:
$$x_0\to x_k\to x_{2k}\to\cdots\to x_{(2l+1)k}.$$
Since $(2l+1)k>nl$ this path performs at least $l$ full revolutions around the circle, hence
$$\sum_{i=0}^{2l}\ddist(x_{ik},x_{(i+1)k})>l.$$
On the other hand
$$\sum_{i=0}^{2l}\ddist(x_{ik},x_{(i+1)k})<(2l+1)r=l+(2l+1)\delta<l+1.$$
It follows that the directed path covers exactly $l$ full circle lengths plus the arc $[x_0,x_{(2l+1)k}]_{S^1}$ whose length, by the last inequality, is less than $(2l+1)\delta$. That proves the claim.
\end{proof}

The results of this section can be summarized as follows. Suppose that $\frac{l}{2l+1}<r<\frac{l+1}{2l+3}$ and $\delta=r-\frac{l}{2l+1}$. Then we know (Proposition~\ref{prop:wfbasic}(c)) that for any finite subset $X\subset S^1$ we have $\wfcless{X}{r}<r<\frac{l+1}{2l+3}$. If we think of $X$ as an evolving (increasing) set, then the homotopy type of $\vrcless{X}{r}$ stabilizes at $S^{2l+1}$ at the same time when $X$ becomes an $\varepsilon$-covering for some $\varepsilon\in[\frac12\delta,(l+\frac12)\delta]$. If $l$ is constant this is a very tight window as $\delta\to 0$.

\section{Evolution of random samples}
\label{sect:evolution}
We now apply the winding fraction to study the evolution of Vietoris--Rips complexes of random subsets of $S^1$. Let $\nicex_n\subseteq S^1$ be a subset obtained by sampling $n$ points uniformly and independently from $S^1$. The connectivity of the graph $\vr{\nicex_n}{r}$ when $r=r(n)\to 0$ as $n\to \infty$ has been extensively studied  by many authors (see \cite{Imany2008} and the references therein). We obtain asymptotic thresholds for the higher-dimensional connectivity of $\vrc{\nicex_n}{r}$ when $r$ is large. In particular, we analyze how many random samples are required until the homotopy type of $\vrc{\nicex_n}{r}$ matches that of $\vrc{S^1}{r}$, extending Latschev's approximation result \cite{Latschev2001} for $S^1$ to $r$ values that are no longer sufficiently small.

In this section we always assume that $l\geq 0$ is fixed and $\frac{l}{2l+1}<r<\frac{l+1}{2l+3}$. We define $\delta=r-\frac{l}{2l+1}$. The probability that two points of $\nicex_n$ are in distance exactly $r$ for any fixed $r$ is $0$, and therefore all results hold for $\mathbf{VR}_<$ as well as $\mathbf{VR}_\leq$. Just as non-trivial asymptotic results about the connectedness of the graph $\vr{\nicex_n}{r}$ can be obtained for $r\to 0$ as $n\to\infty$, in our higher-dimensional regime it makes sense to assume that $r\to\frac{l}{2l+1}$, that is $\delta\to 0$, as $n\to\infty$. We use the standard asymptotic notation $f(\delta)=\Theta(g(\delta))$ as $\delta\to 0$ when there are constants $C_1, C_2>0$ (which can depend on $l$) such that $C_1g(\delta)\leq f(\delta)\leq C_2g(\delta)$.

Let $M(r)$ and $N(r)$ be the random variables counting the number $n$ of random points in $S^1$ until $\wfc{\nicex_n}{r}$ reaches, resp.\ exceeds, the value $\frac{l}{2l+1}$. Formally, consider the random process $(\nicex_1,\nicex_2,\ldots)$ where $\nicex_{i+1}$ is obtained from $\nicex_i$ by adding a single uniformly random point. Define
\begin{eqnarray*}
M(r)=&\min\{~n~:\wfc{\nicex_n}{r}\geq\frac{l}{2l+1}\},\\
N(r)=&\min\{~n~:\wfc{\nicex_n}{r}>\frac{l}{2l+1}\},
\end{eqnarray*}
where $\min\emptyset=\infty$. The random variables $M(r)$ and $N(r)$ describe the last two transition points in the evolution of $\vrc{\nicex_n}{r}$, since $M(r)\le n<N(r)$ means $\vrc{\nicex_n}{r}$ is homotopy equivalent to a wedge of copies of $S^{2l}$, and $n\geq N(r)$ gives $\vrc{\nicex_n}{r}\htpyequiv S^{2l+1}$. We will determine the asymptotic expectations $\expect[M(r)]$ and $\expect[N(r)]$.

\begin{theorem}
\label{thm:thresholds}
Let $\frac{l}{2l+1}<r<\frac{l+1}{2l+3}$ for some fixed $l\geq 0$ and let $\delta=r-\frac{l}{2l+1}$. We have
$$
\expect[M(r)]=\Theta\left(\left(\frac{1}{\delta}\right)^{\frac{2l}{2l+1}}\right),\quad \expect[N(r)]=\Theta\left(\frac{1}{\delta}\log\frac{1}{\delta}\right)\quad \mathrm{as}\ \delta\to 0.
$$
In particular, the expected number of random points $n$ until $\vrc{\nicex_n}{r}\htpyequiv S^{2l+1}$ is $\Theta(\frac{1}{\delta}\log\frac{1}{\delta})$ as $\delta\to 0$.
\end{theorem}
Note that the winding fraction of $\frac{l}{2l+1}$ is achieved much sooner than it is exceeded (in fact $\expect[M(r)]$ is sublinear in $1/\delta$). It means that we are expecting a long interval of $n$ for which $\vrc{\nicex_n}{r}$ is a wedge of $2l$-spheres, before reaching the final homotopy type of $S^{2l+1}$. 

\begin{example}
\label{ex:randomexample}
Suppose $\frac{3}{7}<r=0.432<\frac{4}{9}$ with $l=3$, $\delta\approx 0.00343$, and $1/\delta\approx 291$. Figure~\ref{fig:randomexample} shows the average evolution of $\vrc{\nicex_n}{r}$ for $1\leq n\leq 1000$. The red curve plots the average winding fraction, which rapidly approaches $3/7$ and then exceeds it around $n=600$ to approach $r$. The homotopy type then stabilizes at $S^7$. For clarity of the presentation the blue curve depicts the average \emph{intrinsic dimension}, which we define as $2l$ when $\wf(\cdot)=\frac{l}{2l+1}$ and as $2l+1$ when $\frac{l}{2l+1}<\wf(\cdot)<\frac{l+1}{2l+3}$.
\end{example}

\begin{figure}
\includegraphics[scale=1]{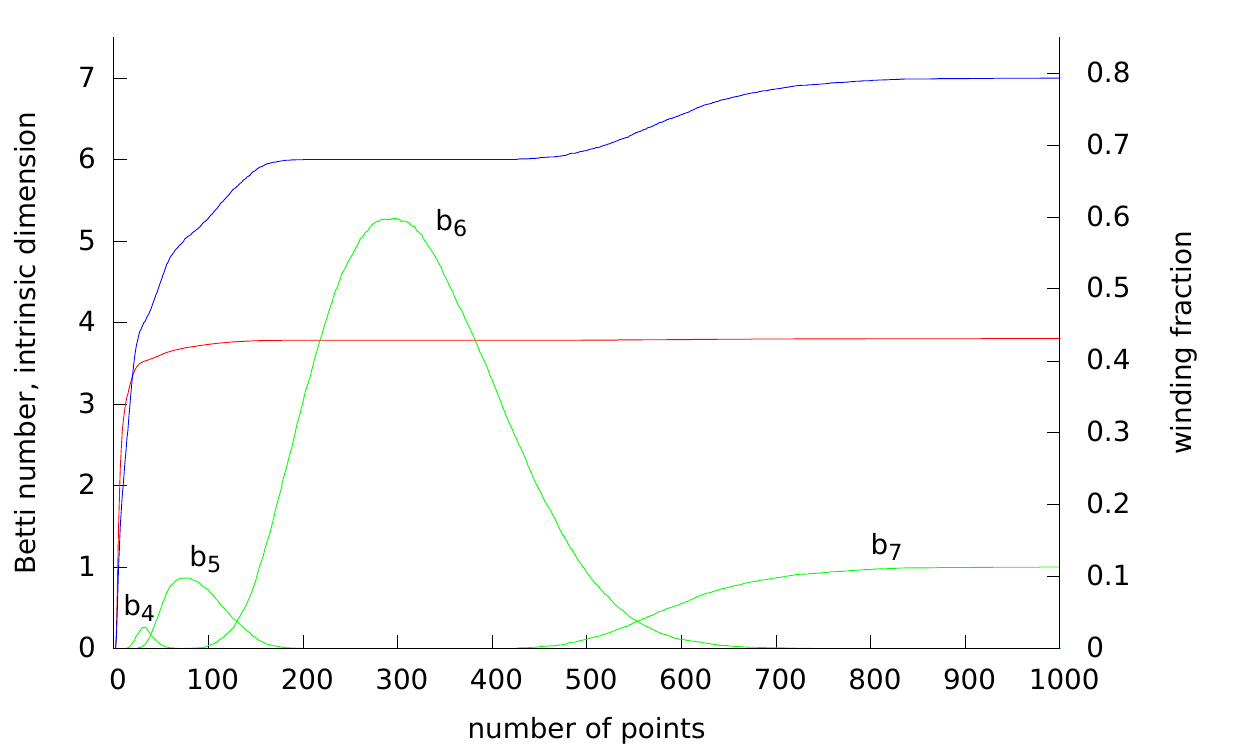}
\caption{The evolution of $\vrc{\nicex_n}{r}$ with $r=0.432$; see Example~\ref{ex:randomexample}. The red curve is the average winding fraction, the blue curve is the average intrinsic dimension, and the green curves are the average Betti numbers $b_4$, $b_5$, $b_6$, $b_7$ (from left to right). Note the support of $b_6$ (for example) mostly coincides with the average intrinsic dimension being close to 6.}
\label{fig:randomexample}
\end{figure}

We first prove the second claim of Theorem~\ref{thm:thresholds}. For $\varepsilon>0$ let $C(\varepsilon)$ be the random variable which counts the number of steps until $\nicex_n$ becomes a $(\frac12\varepsilon)$-covering of $S^1$. By Propositions~\ref{prop:approachinghighwf} and \ref{prop:highwfisdense} we have
\begin{equation}
\label{eq:twosideddensity}
C((2l+1)\delta)\leq N(r)\leq C(\delta).
\end{equation} 
It is well-known that 
\begin{equation}
\label{eq:expectC}
\expect[C(\varepsilon)]=\Theta\left(\varepsilon^{-1}\log\varepsilon^{-1}\right)
\end{equation}
 as $\varepsilon\to 0$ (see \cite[Equation (4.16)]{Solomon}, which gives a more precise answer). The asymptotics of \eqref{eq:expectC} can also be seen heuristically as follows. Divide $S^1$ into $K=\Theta(\varepsilon^{-1})$ arcs of length $\Theta(\varepsilon)$. Think of the random process $\nicex_n$ as throwing balls into $K$ urns (arcs) independently. Then the event that $\nicex_n$ is a $\varepsilon$-covering coincides with the event that each urn contains a ball. By the classical coupon collector's problem this happens, in expectation, after $n=\Theta(K\log K)$ balls as $K\to\infty$. Combining \eqref{eq:expectC} with \eqref{eq:twosideddensity} gives $\expect[N(r)]=\Theta(\delta^{-1}\log\delta^{-1})$ as $\delta\to 0$.

\medskip
To prove the first statement of Theorem~\ref{thm:thresholds} we need some auxiliary results. A subset $Y\subset S^1$ will be called \emph{$(\varepsilon,m)$-regular} if $|Y|=m$ and there is a bijection from $Y$ to the vertices of some regular inscribed $m$-gon which moves each point by distance less than $\epsilon$. We previously showed that achieving $\wfc{X}{r}>\frac{l}{2l+1}$ coincides with $X$ being a  $\Theta(\delta)$-covering, and the next lemma shows that $\wfc{X}{r}\geq \frac{l}{2l+1}$ is achieved when $X$ contains a $(\Theta(\delta),2l+1)$-regular subset.
\begin{lemma}
\label{lem:almostregular}
Let $\frac{l}{2l+1}<r<\frac{l+1}{2l+3}$ and $\delta=r-\frac{l}{2l+1}$. For a finite subset $X\subseteq S^1$ we have 
\begin{itemize}
\item[(a)] If $X$ has a $(\frac12\delta, 2l+1)$-regular subset then $\wfc{X}{r}\geq \frac{l}{2l+1}$.
\item[(b)] If $\wfc{X}{r}\geq \frac{l}{2l+1}$ then $X$ has a $(4l\delta, 2l+1)$-regular subset.
\end{itemize}
\end{lemma}
\begin{proof}
For (a) let $\{x_0,\ldots,x_{2l}\}\subset X$ be the $(\frac12\delta,2l+1)$-regular subset. We can assume $x_i\in (\frac{i}{2l+1}-\frac12\delta,\frac{i}{2l+1}+\frac12\delta)_{S^1}$. Since $\delta<\frac{1}{2l+1}$ we have $x_0\prec x_1\prec\cdots\prec x_{2l}$ cyclically ordered in $S^1$ and in $\vrdir{X}{r}$. We have
$$\ddist(x_i,x_{i+l})<\tfrac{l}{2l+1}+2\cdot\tfrac12\delta=r$$
and hence a cyclic homomorphism $\cnkdir{2l+1}{l}\incl \vrdir{X}{r}$.

To prove (b) suppose $\wfc{X}{r}\geq\frac{l}{2l+1}$. By Proposition~\ref{prop:cychomoall} there is a directed homomorphism $f\colon \cnkdir{2l+1}{l}\to\vrdir{X}{r}$. Denote $x_i=f(i)$. For every $i=0,\ldots,2l$ we get
$$\ddist(x_i,x_{i+1})=1-\ddist(x_{i+1},x_{i+l+1})-\ddist(x_{i+l+1},x_i)>1-2r=\tfrac{1}{2l+1}-2\delta.$$
It follows that for $j=1,\ldots,2l$ we have
$$\ddist(x_i,x_{i+j})> \tfrac{j}{2l+1}-2j\delta\geq \tfrac{j}{2l+1}-4l\delta$$
and in turn
$$\ddist(x_i,x_{i+j})=1-\ddist(x_{i+j},x_i)< 1-\tfrac{2l+1-j}{2l+1}+4l\delta=\tfrac{j}{2l+1}+4l\delta.$$
It follows that each $x_j$ lies in distance less than $4l\delta$ from the $j$-th vertex of the regular $(2l+1)$-gon with $x_0$ as a vertex. 
\end{proof}

Let $R_m(\varepsilon)$ be the random variable which counts the number of steps until $\nicex_n$ contains a $(\varepsilon,m)$-regular subset. In Appendix~\ref{sect:appendixB} we show that for every fixed $m$ 
\begin{equation}
\label{eq:expectR}
\expect[R_m(\varepsilon)]=\Theta(\varepsilon^{-\frac{m-1}{m}})\quad \mathrm{as}\ \varepsilon\to 0.
\end{equation}
Here we only give a heuristic explanation using our previous urn model with $K=\Theta(\varepsilon^{-1})$ urns identified with arcs of length $\Theta(\varepsilon)$. Divide the urns  into $K/m$ groups of size $m$, each group consisting of arcs centered approximately around the vertices of a regular $m$-gon. Then the event that $\nicex_n$ has an $(\varepsilon,m)$-regular subset coincides with the event that every urn in some group contains a ball. This can be correlated with the generalized birthday paradox, where we require one urn to contain $m$ balls (the case $m=2$ is the classical birthday paradox). The expected waiting time for this to happen is $\Theta(K^{\frac{m-1}{m}})$ as $K\to\infty$; see \cite[Theorem 2]{KlamkinNewman1967}.

This proves the first claim of Theorem~\ref{thm:thresholds} since by Lemma~\ref{lem:almostregular} we have
$$R_{2l+1}(4l\delta)\leq N(r)\leq R_{2l+1}(\tfrac12\delta).$$
The proof of Theorem~\ref{thm:thresholds} is now complete.

\section{Vietoris-Rips complexes for subsets of $S^1$}
\label{sect:s1}

The definition of the Vietoris--Rips complex $\vrc{X}{r}$ makes sense for an arbitrary metric space $X$, not necessarily finite nor discrete. In \cite{Hausmann1995}, Hausmann studied the case when $X$ is a closed Riemannian manifold. In this section we show that for an arbitrary subset $X\subseteq S^1$ and $r>0$, the complex $\vrc{X}{r}$ has the homotopy type of an odd-dimensional sphere or a wedge of even-dimensional spheres. We will also study the complexes $\vrcless{S^1}{r}$ and $\vrcleq{S^1}{r}$ in more detail.

For an arbitrary metric space $X$ the geometric realization of $\vrc{X}{r}$ is given the topology of a CW-complex, that is the weak topology with respect to finite-dimensional skeleta, or equivalently, the weak topology with respect to subcomplexes induced by finite subsets of $X$. Formally, let $F(X)$ be the poset of all finite subsets of $X$ ordered by inclusion. Then for each $r$ we have a functor $\vrc{-}{r}: F(X)\to Top$ and
\begin{equation*}
\label{eq:colimit}
\vrc{X}{r} = \colim_{Y\in F(X)}\vrc{Y}{r} \htpyequiv \hocolim_{Y\in F(X)} \vrc{Y}{r},
\end{equation*}
where the last equivalence is a consequence of the fact that all maps $\vrc{Y}{r}\incl\vrc{Y'}{r}$ for $Y\subseteq Y'$ are inclusions of closed subcomplexes, hence cofibrations. See \cite[Section 3]{Hocolim} for the statements of all diagram comparison theorems used in this section.

For a finite subset $Y_0\subseteq X$ let $F(X;Y_0)$ be the subposet of $F(X)$ consisting of all sets which contain $Y_0$. Since this poset is cofinal in $F(X)$, we also have
\begin{equation*}
\label{eq:colimit2}
\vrc{X}{r} = \colim_{Y\in F(X;Y_0)}\vrc{Y}{r} \htpyequiv \hocolim_{Y\in F(X;Y_0)} \vrc{Y}{r}.
\end{equation*}

For an arbitrary subset $\emptyset\neq X\subseteq S^1$ and $0<r<\frac12$ we define
\begin{equation}
\label{eq:wfsubset}
\begin{aligned}
\wfcless{X}{r}&=\sup\{\wfcless{Y}{r}~:~Y\subseteq X, |Y|<\infty\},\\
\wfcleq{X}{r}&=\sup\{\wfcleq{Y}{r}~:~Y\subseteq X, |Y|<\infty\}.
\end{aligned}
\end{equation}
The supremum in \eqref{eq:wfsubset} need not be attained when $X$ is infinite. When $X$ is finite then $\wfc{X}{r}$ agrees with our previous definition of this symbol since the supremum is attained by $Y=X$. The following proposition shows that in the generic case, $\vrc{X}{r}$ has the homotopy type of an odd-dimensional sphere.

\begin{proposition}
\label{prop:subsetsofs1odd}
Suppose that $\emptyset\neq X\subseteq S^1$ and $0<r<\frac12$. Either of the two conditions
\begin{itemize}
\item[(1)] $\frac{l}{2l+1}<\wfc{X}{r}<\frac{l+1}{2l+3}$, or
\item[(2)] $\wfc{X}{r}=\frac{l+1}{2l+3}$ and the supremum is not attained,
\end{itemize}
for some $l=0,1,\ldots$, implies that $\vrc{X}{r}\htpyequiv S^{2l+1}$.

Moreover, if $r'\geq r$ is another value of the distance parameter for which (1) or (2) hold with the same $l$, then the inclusion $\vrc{X}{r}\incl\vrc{X}{r'}$ is a homotopy equivalence.
\end{proposition}
\begin{proof}
Either of the two conditions (1), (2) implies there is a finite subset $Y_0\subseteq X$ such that for every finite subset $Y$ with $Y_0\subseteq Y\subseteq X$, we have $\frac{l}{2l+1}<\wfc{Y}{r}<\frac{l+1}{2l+3}$. By Proposition~\ref{prop:inducedhtpyequiv} all maps in the diagram $\vrc{-}{r}: F(X;Y_0)\to Top$ are homotopy equivalences between spaces homotopy equivalent to $S^{2l+1}$, and therefore
$$
\vrc{X}{r}\htpyequiv \hocolim_{Y\in F(X,Y_0)} \vrc{Y}{r}\htpyequiv S^{2l+1}.
$$
Furthermore we have $\frac{l}{2l+1}<\wfc{Y}{r}\leq\wfc{Y}{r'}<\frac{l+1}{2l+3}$, hence the same is true for $\vrc{X}{r'}$. The maps $\vrc{Y}{r}\to\vrc{Y}{r'}$ now define a natural transformation of diagrams $\vrc{-}{r}\to\vrc{-}{r'}$ which is a levelwise homotopy equivalence by Proposition~\ref{prop:inducedhtpyequiv}. It follows that the induced map of (homotopy) colimits is a homotopy equivalence.
\end{proof}
\begin{remark}
\label{rem:finiteapproximation}
The same argument shows that under the assumptions of the last proposition the map $\vrc{Y_0}{r}\incl\vrc{X}{r}$ is a homotopy equivalence whenever $Y_0\subseteq X$ is a finite set with $\wfc{Y_0}{r}>\frac{l}{2l+1}$.
\end{remark}

As the next lemma shows, the winding fractions behave in the expected way for dense subsets of the circle.
\begin{lemma}
\label{lem:wfdenseA}
If $X$ is dense in $S^1$ and $0<r<\frac12$ then $\wfcless{X}{r}=\wfcleq{X}{r}=r$. In the case of $\wf_<$ the supremum is not attained.
\end{lemma}
\begin{proof}
For every $\varepsilon>0$ the set $X$ contains a finite $\varepsilon$-covering of $S^1$. Proposition~\ref{prop:approachinghighwf} now gives $\wfc{X}{r}\geq r$. The reverse inequality and the second statement of the lemma follow from Proposition~\ref{prop:wfbasic}(c).
\end{proof}

We can now give a complete description of the homotopy types of $\vrcless{S^1}{r}$ for arbitrary $r$.

\begin{theorem}
\label{thm:vrlessfull}
If $X$ is dense in $S^1$ (in particular when $X=S^1$) and $0<r<\frac12$, then we have
$$\vrcless{X}{r}\htpyequiv S^{2l+1} \quad \mathrm{for}\ \tfrac{l}{2l+1}<r\leq\tfrac{l+1}{2l+3}, \ l=0,1,\ldots.$$
Moreover, if $\frac{l}{2l+1}<r\leq r'\leq\frac{l+1}{2l+3}$ then the inclusion $\vrcless{X}{r}\incl\vrcless{X}{r'}$ is a homotopy equivalence.
\end{theorem}
\begin{proof}
By Lemma~\ref{lem:wfdenseA} we have $\wfcless{X}{r}=r$ and the supremum is not attained, meaning that either (1) or (2) in Proposition~\ref{prop:subsetsofs1odd} is satisfied. 
\end{proof}

\medskip
Proposition~\ref{prop:subsetsofs1odd} describes the homotopy types of $\vrc{X}{r}$ in all generic situations. The only singular cases it does not cover occur when $\wf(X;r)$ is of the form $l/(2l+1)$ and this value is in fact attained by some finite subset $Y_0\subseteq X$. We deal with this in the next two statements.

\begin{proposition}
\label{prop:vrc-even}
Suppose that $\emptyset\neq X\subseteq S^1$ and $0<r<\frac12$. If $\wfc{X}{r}=\frac{l}{2l+1}$ for some $l=0,1,\ldots$ and the supremum in the definition of $\wfc{X}{r}$ is attained, then $\vrc{X}{r}$ is homotopy equivalent to a wedge sum of spheres of dimension $2l$.
\end{proposition}

\begin{theorem}
\label{thm:vrcleqfull}
For $0\leq r<\frac12$ we have a homotopy equivalence
$$
\vrcleq{S^1}{r}\htpyequiv \begin{cases}
S^{2l+1} & \mathrm{if}\ \frac{l}{2l+1}<r<\frac{l+1}{2l+3},\ l=0,1,\ldots, \\
\bigvee^{\mathfrak{c}} S^{2l} & \mathrm{if}\ r=\frac{l}{2l+1},
\end{cases}
$$
where $\mathfrak{c}$ is the cardinality of the continuum. Moreover, if $\frac{l}{2l+1}<r\leq r'<\frac{l+1}{2l+3}$ then the inclusion $\vrcleq{S^1}{r}\incl\vrcleq{S^1}{r'}$ is a homotopy equivalence.
\end{theorem}

We delay the proofs of Proposition~\ref{prop:vrc-even} and Theorem~\ref{thm:vrcleqfull} until Section~\ref{sect:singular}. Note that Theorems~\ref{thm:vrlessfull} and \ref{thm:vrcleqfull} together provide a complete description of the homotopy types of $\vrc{S^1}{r}$ for arbitrary $r$. They also give the persistent homology of $\vrc{S^1}{r}$, where we refer the reader to \cite{ChazalDeSilvaOudot2013} for information on the persistent homology of Vietoris--Rips complexes.

\begin{corollary}
\label{cor:vrcpersistenthomology}
The persistent homology of $\vrcless{S^1}{r}$ contains a single interval $(\frac{l}{2l+1},\frac{l+1}{2l+3}]$ in each homological dimension $2l+1$, and the persistent homology of $\vrcleq{S^1}{r}$ contains a single interval $(\frac{l}{2l+1},\frac{l+1}{2l+3})$ in each homological dimension $2l+1$.
\end{corollary}

\begin{remark}
\label{rem:provehausmann}
Hausmann \cite[(3.12)]{Hausmann1995} conjectured that if $M$ is a compact Riemannian manifold then the connectivity $\conn(\vrcless{M}{r})$ is a non-decreasing function of $r$, and our results confirm this conjecture for $M=S^1$. Hausmann \cite[(3.11)]{Hausmann1995} furthermore conjectured that for $r$ sufficiently small, $\vrcless{M}{r}$ is homotopy equivalent to $\vrcless{Y_0}{r}$ for some finite subset $Y_0\subset M$. For $M=S^1$ we confirm this conjecture for all $r$, sufficiently small or otherwise.
\end{remark}

\section{Singular winding fractions}
\label{sect:singular}

In this section we return to study cyclic graphs for which $\wf(\dirg{G})=\frac{l}{2l+1}$ is a singular value. Our aim is to  describe a convenient structure in the homology group $\redhom_{2l}(\cl(G))$, which we then use to prove Proposition~\ref{prop:vrc-even} and Theorem~\ref{thm:vrcleqfull}.

We consider first a cyclic graph $\cnkdir{n}{k}$ with $\frac{k}{n}=\frac{l}{2l+1}$. Since $l$ and $2l+1$ are coprime we have $(k,n)=(dl,d(2l+1))$ for some integer $d\geq 1$. We have $d=n-2k$ and so by Theorem~\ref{thm:cyclichtpy} we can write $\cl(\cnk{d(2l+1)}{dl})\htpyequiv\bigvee^{d-1}S^{2l}$. When $(k,n)=(l,2l+1)$ the graph $\cnk{2l+1}{l}$ is a clique and $\cl(\cnk{2l+1}{l})$ is the full simplex with $2l+1$ vertices.

The next case, $d=2$ and $(k,n)=(2l,2(2l+1))$, is particularly interesting for our purposes. The non-edges of the graph $\cnk{2(2l+1)}{2l}$ are pairs of the form $\{i,i+2l+1\}$, which are the antipodal pairs in the evenly-spaced model $\cnk{2(2l+1)}{2l}=\vrleq{\{\frac{i}{2(2l+1)}~:~i=0,\ldots,4l+1\}}{\frac{l}{2l+1}}$. It follows that the clique complex $\cl(\cnk{2(2l+1)}{2l})$ is isomorphic to the standard triangulation of $S^{2l}$ as the boundary of the cross-polytope of dimension $2l+1$. We fix the $2l$-dimensional cycle in $\cl(\cnk{2(2l+1)}{2l})$:
\begin{align}
\label{eq:iota-cycle}
\iota_{2l}&=(-1)^{l(l+3)/2}\cdot([0]-[2l+1])\wedge([1]-[2l+2])\wedge\cdots\wedge([2l]-[4l+1])\nonumber\\
&=[0,2,\ldots,4l]-[1,3,\ldots,4l+1]\pm\cdots,
\end{align}
which is (up to sign) the fundamental cycle of the boundary of the cross-polytope. Here $[x_0]\wedge\dots\wedge[x_k]$ denotes the oriented simplex $[x_0, \ldots, x_k]$, and we have chosen the sign so that the oriented simplices $[0,2,\ldots,4l]$ and $[1,3,\ldots,4l+1]$ appear with coefficients $+1$ and $-1$ respectively. Indeed, in $([0]-[2l+1])\wedge\cdots\wedge([2l]-[4l+1])$ the sign on $[0,2l+2,2,\ldots,2l-2,4l,2l]$ is $(-1)^l$, and then after $l(l+1)/2$ transpositions this gives the sign $(-1)^{l(l+3)/2}$ on $[0,2,\ldots,4l]$. The argument for $[1,3,\ldots,4l+1]$ is similar. The corresponding homology class $\iota_{2l}\in \redhom_{2l}(\cl(\cnk{2(2l+1)}{2l}))=\ZZ$ is a generator. (Here and in the following we will use the same symbol to denote a (co)cycle and its (co)homology class, and sometimes also the map which induces the given class.)

\begin{definition}
\label{def:crosspoly-class}
Suppose that $\dirg{G}$ is a cyclic graph. A non-zero homology class $\alpha\in \redhom_{2l}(\cl(G))$ is called \emph{cross-polytopal} if there is a cyclic homomorphism $f\colon\cnkdir{2(2l+1)}{2l}\to \dirg{G}$ such that $\alpha=f_*(\iota_{2l})$.
\end{definition}
An immediate consequence of the definition is that the image of a cross-polytopal class under a cyclic homomorphism $\dirg{G}\to\dirg{H}$ is again cross-polytopal, unless it is zero. Note that if $f$ is not injective on the vertices then $f_*(\iota_{2l})=0$ because a homology class of degree $2l$ in a clique complex must be supported on at least $4l+2$ vertices (see for instance \cite[Lemma 5.3]{Kahle2009}).
 
Our aim is to classify all cross-polytopal homology classes for cyclic graphs. We begin with the description of a class of cyclic homomorphisms.

\begin{lemma}
\label{lem:crosspoly-homo}
Let $d\geq 1$ and $(k,n)=(dl,d(2l+1))$.
\begin{itemize}
\item[(a)] Every cyclic homomorphism $\cnkdir{2l+1}{l}\to\cnkdir{n}{k}$ is of the form $\theta_a$ for some $a=0,\ldots,n-1$, where
$$\theta_a(i)=a+di \md{n}.$$
\item[(b)] Every injective cyclic homomorphism $\cnkdir{2(2l+1)}{2l}\to\cnkdir{n}{k}$ is of the form $\alpha_{a,b}$ for some $a=0,\ldots,n-1$ and $b=a+1,\ldots,a+d-1$, where
$$\alpha_{a,b}(i)=\begin{cases}a+d\cdot\frac{i}{2} \md{n} & \mathrm{if}\ i\ \mathrm{is}\ \mathrm{even},\\b+d\cdot\frac{i-1}{2} \md{n}& \mathrm{if}\ i\ \mathrm{is}\ \mathrm{odd}.\end{cases}$$
\end{itemize}
\end{lemma}
\begin{remark}
Every cyclic homomorphism $\theta$ in part (a) is determined by the choice of $a=\theta(0)$ and the condition $\theta(i+1)=\theta(i)+d \md{n}$. Similarly, in part (b) every cyclic homomorphism is determined by the two initial values $a=\alpha(0)$ and $b=\alpha(1)$, together with the requirement that $\alpha(i+2)=\alpha(i)+d \md{n}$.
\end{remark}
\begin{proof}
To prove (a) let $\theta:\cnkdir{2l+1}{l}\to\cnkdir{n}{k}$ be a cyclic homomorphism with $l>0$, since the case $l=0$ is clear. Then we have
\begin{equation*}
(2l+1)k\geq \sum_{i=0}^{2l}\ddist_n(\theta(i),\theta(i+l))=l\cdot\sum_{i=0}^{2l}\ddist_n(\theta(i),\theta(i+1))=ln,
\end{equation*}
where the last equality follows from \eqref{eq:degree1vr}. Since the two extremes are in fact equal, we must have $\ddist_n(\theta(i),\theta(i+l))=k$ for all $i$, which implies
$$\ddist_n(\theta(i),\theta(i+1)) = n-\ddist_n(\theta(i+1),\theta(i+l+1))-\ddist_n(\theta(i+l+1),\theta(i))=n-2k=d$$
as required. Clearly every $\theta_a$ is a cyclic homomorphism, hence (a) is proved.

Part (b) follows immediately, since $\cnkdir{2(2l+1)}{2l}$ contains two induced copies of $\cnkdir{2l+1}{l}$ with vertex sets $\{0,2,\ldots,4l\}$ and $\{1,3,\ldots,4l+1\}$. For an injective cyclic homomorphism $\alpha$ we must have $\alpha(0)\prec\alpha(1)\prec\alpha(2)$ cyclically ordered in $\cnkdir{n}{k}$, i.e.\ $a\prec b\prec a+d$, which yields the restrictions on $b$.
\end{proof}

The cyclic homomorphism $\alpha_{a,b}$ evaluated on the fundamental cycle $\iota_{2l}$ determines a cycle and homology class in $\redhom_{2l}(\cl(\cnk{d(2l+1)}{dl}))$. We will continue to denote both with $\alpha_{a,b}$. The chain representation of the cycle $\alpha_{a,b}$ starts with
\begin{equation}
\label{eq:alpha-cycle}
\alpha_{a,b} = [a,a+d,\ldots,a+2l\cdot d]-[b,b+d,\ldots,b+2l\cdot d]\pm\cdots;
\end{equation}
compare \eqref{eq:iota-cycle}. The homology classes $\alpha_{a,b}$ for various pairs $(a,b)$ satisfy a number of relations worked out in the proof of the next proposition, which is the main result concerning cyclic graphs with $\wf(\dirg{G})=\frac{l}{2l+1}$.

\begin{proposition}
\label{prop:crosspoly-classify}
Suppose $\dirg{G}$ is a cyclic graph which dismantles to $\cnkdir{d(2l+1)}{dl}$. Then the homology group $\redhom_{2l}(\cl(G))=\ZZ^{d-1}$ has a basis $\{e_1,\ldots,e_{d-1}\}$ such that all the cross-polytopal elements in $\redhom_{2l}(\cl(G))$ are
$$\pm\ e_1,\ldots, \pm\ e_{d-1}$$
and
$$e_i-e_j, \quad 1\leq i,j\leq d-1,\ i\neq j.$$
In particular, there are exactly $d(d-1)$ cross-polytopal elements in $\redhom_{2l}(\cl(G))$.
\end{proposition}
\begin{proof}
In the first step we will prove the result for $\dirg{G}=\cnkdir{d(2l+1)}{dl}$. Denote $(k,n)=(dl,d(2l+1))$. 

For an oriented simplex $\sigma$ in a simplicial complex $K$ let $\sigma^\vee$ denote the cochain which assigns $1$ to $\sigma$, $-1$ to $\sigma$ with opposite orientation, and $0$ to all other oriented simplices of $K$. For every $a=0,\ldots,n-1$ define a cochain $\beta_a$ in $\cl(\cnk{n}{k})$ by
\begin{equation*}
\label{eq:beta-def}
\beta_a = [a,a+d,\ldots,a+2l\cdot d]^\vee.
\end{equation*}
Since the face $[a,a+d,\ldots,a+2l\cdot d]$ is maximal in $\cl(\cnk{n}{k})$, the cochain $\beta_a$ is in fact a cocycle, and it determines a cohomology class which we denote with the same symbol. Using \eqref{eq:alpha-cycle} we verify that for $1\leq i,j\leq d-1$ 
\begin{equation*}
\label{eq:dual-pairing}
\beta_i(\alpha_{j,d})=\begin{cases}1 & \mathrm{if}\ i=j, \\ 0 & \mathrm{if}\ i\neq j.\end{cases}
\end{equation*}
Since the groups $\redhom_{2l}(\cl(\cnk{n}{k}))$ and $\redhom^{2l}(\cl(\cnk{n}{k}))$ are both free abelian of rank $d-1$, the above implies that $\{\alpha_{1,d},\ldots,\alpha_{d-1,d}\}$ is a basis of homology and $\{\beta_1,\ldots,\beta_{d-1}\}$ is its dual basis of cohomology. In particular, every element $v\in \redhom_{2l}(\cl(\cnk{n}{k}))$ has a decomposition
\begin{equation}
\label{eq:decomposition}
v=\sum_{i=1}^{d-1}\beta_i(v)\cdot\alpha_{i,d}.
\end{equation}
Note that $\beta_i(\alpha_{a,b})$ depends only on the evaluation of $\beta_i$ on the two leading terms in \eqref{eq:alpha-cycle}, since $\beta_i$ evaluates to $0$ on all the omitted terms. The oriented simplices appearing in $\alpha_{a,b}$ and $\alpha_{a+d,b+d}$ differ by a cyclic shift, hence by an even number of $2l$ transpositions, and are therefore equal. That means we have the identity
$$\alpha_{a+d,b+d}=\alpha_{a,b}.$$
It follows that all cross-polytopal classes can be written as $\alpha_{a,b}$ with $0\leq a\leq d-1$ and $a+1\leq b\leq a+d-1$.

If $a=0$ then $1\leq b\leq d-1$ and the only non-zero pairing in \eqref{eq:decomposition} is $\beta_b(\alpha_{0,b})=-1$, and hence $\alpha_{0,b}=-\alpha_{b,d}$.

If $1\leq a<b\leq d-1$ then $\beta_a(\alpha_{a,b})=1$ and $\beta_b(\alpha_{a,b})=-1$; hence $\alpha_{a,b}=\alpha_{a,d}-\alpha_{b,d}$.

If $1\leq a\leq d-1$ and $b=d$ then $\alpha_{a,b}=\alpha_{a,d}$ is itself one of the generators.

If $1\leq a\leq d-1$ and $d+1\leq b\leq a+d-1$ then $1\leq b-d<a\leq d-1$. Using the cyclic shift argument we obtain $\beta_a(\alpha_{a,b})=1$ and $\beta_{b-d}(\alpha_{a,b})=-1$, hence $\alpha_{a,b}=\alpha_{a,d}-\alpha_{b-d,d}$.

It follows that the proposition is true with $e_i=\alpha_{i,d}$ for $i=1,\ldots,d-1$.

\medskip
Now suppose $\dirg{G}$ is an arbitrary cyclic graph which dismantles to $\cnkdir{n}{k}$. By Corollary~\ref{cor:dismantlinghtpy} the cyclic homomorphisms $\cnkdir{n}{k}\xrightarrow{\iota} \dirg{G}\xrightarrow{\pi}\cnkdir{n}{k}$ induce isomorphisms $\redhom_{2l}(\cl(\cnk{n}{k}))\xrightarrow{\cong}\redhom_{2l}(\cl(G))\xrightarrow{\cong}\redhom_{2l}(\cl(\cnk{n}{k}))$ with the composition being the identity. It follows that the cross-polytopal classes $\iota_*(e_1),\ldots,\iota_*(e_{d-1})$ form a basis of $\redhom_{2l}(\cl(G))$ and that $\pm\iota_*(e_i)$, $\iota_*(e_i)-\iota_*(e_j)$, $i\neq j$, are cross-polytopal. Moreover, if $\alpha\in\redhom_{2l}(\cl(G))$ is cross-polytopal then $\pi_*(\alpha)$ is one of $\pm e_i$, $e_i-e_j$, $i\neq j$, and therefore $\alpha$ must be one of $\pm\iota_*(e_i)$, $\iota_*(e_i)-\iota_*(e_j)$, $i\neq j$. That completes the proof.
\end{proof}
 
 \medskip
 
We are now prepared to prove Proposition~\ref{prop:vrc-even}, using the algebraic fact in Proposition~\ref{prop:appendixA} of the appendix.

\begin{proof}[Proof of Proposition~\ref{prop:vrc-even}]
Let $Y_0\subseteq X$ be a finite subset which achieves $\wfc{Y_0}{r}=\frac{l}{2l+1}$. Then we have $\wfc{Y}{r}=\frac{l}{2l+1}$ for any finite subset $Y$ with $Y_0\subseteq Y\subseteq X$. By Corollary~\ref{cor:vrchtpy} every space in the diagram
$$\vrc{X}{r} = \colim_{Y\in F(X;Y_0)}\vrc{Y}{r}$$
is homotopy equivalent to a finite wedge sum of $2l$-spheres. It follows immediately that $\vrc{X}{r}$ is simply-connected and its homology is torsion-free and concentrated in degree $2l$. It remains to show that the group $\redhom_{2l}(\vrc{X}{r})$ is free abelian. Indeed, if this is the case then $\vrc{X}{r}$ is a model of the Moore space $M(\bigoplus^\kappa \ZZ,2l)$, unique up to homotopy and equivalent to $\bigvee^\kappa S^{2l}$, for some cardinal number $\kappa$.

A nonzero homology class in $\redhom_{2l}(\vrc{X}{r})$ will be called \emph{cross-polytopal} if it is the image under the inclusion $\vrc{Y}{r}\incl\vrc{X}{r}$ of a cross-polytopal class in $\redhom_{2l}(\vrc{Y}{r})$ for some finite $Y\in F(X;Y_0)$. Since the groups $\redhom_{2l}(\vrc{Y}{r})$ are generated by cross-polytopal classes, the same is true about their colimit, $\redhom_{2l}(\vrc{X}{r})$. 

A subset $B$ of an abelian group $G$ is called \emph{independent} if for every finite subset $\{b_1,\ldots,b_s\}\subseteq B$ the identity $\sum_{i=1}^s a_ib_i=0$, with $a_1,\ldots,a_s\in\ZZ$, implies $a_1=\cdots=a_s=0$. An independent set $B$ generates a free abelian subgroup of $G$ with basis $B$. Now let $\niceb$ be the family of all subsets $B\subseteq \redhom_{2l}(\vrc{X}{r})$ such that
\begin{itemize}
\item[(a)] all elements of $B$ are cross-polytopal,
\item[(b)] $B$ is independent.
\end{itemize}
The family $\niceb$ is nonempty and closed under increasing unions. Using Zorn's lemma pick an inclusion-wise maximal set $B$ satisfying (a) and (b). If $B$ generates $\redhom_{2l}(\vrc{X}{r})$ then we are done, because the group $\langle B\rangle$ generated by $B$ is free abelian. 

We suppose for a contradiction that $B$ does not generate $\redhom_{2l}(\vrc{X}{r})$, and hence there exists a cross-polytopal class $v\not\in\langle B\rangle$ since $\redhom_{2l}(\vrc{X}{r})$ is generated by cross-polytopal classes. By maximality of $B$ the set $B\cup\{v\}$ violates (b), and hence there exists a non-trivial linear relation involving $v$ and a finite number of elements $b_1,\ldots,b_s\in B$. In other words, some non-trivial multiple of $v$ lies in the subgroup of $\redhom_{2l}(\vrc{X}{r})$ generated by $b_1,\ldots,b_s$. The same relation holds for the cross-polytopal representatives $v, b_1,\ldots,b_s$ at some finite stage $\redhom_{2l}(\vrc{Y}{r})$ of the colimit, where $\vrc{Y}{r}$ dismantles to $\cnkdir{d(2l+1)}{dl}$. Changing signs if necessary we may assume that each of the elements $v,b_1,\ldots,b_s\in\redhom_{2l}(\vrc{Y}{r})=\ZZ^{d-1}$ is of the form $e_i$ or $e_i-e_j$, $i<j$, for the basis $\{e_1,\ldots,e_{d-1}\}$ from Proposition~\ref{prop:crosspoly-classify}. Now Proposition~\ref{prop:appendixA} implies that $v$ itself lies in the subgroup of $\redhom_{2l}(\vrc{Y}{r})$ generated by $b_1,\ldots,b_s$, and hence in the subgroup of $\redhom_{2l}(\vrc{X}{r})$ generated by $b_1,\ldots,b_s$. This contradiction shows that in fact $\redhom_{2l}(\vrc{X}{r})=\langle B\rangle$ is free abelian.
\end{proof}

The last item in this section is the proof of Theorem~\ref{thm:vrcleqfull}.

\begin{proof}[Proof of Theorem~\ref{thm:vrcleqfull}]
By Lemma~\ref{lem:wfdenseA} we have $\wfcleq{S^1}{r}=r$, and so all statements concerning the generic values of $r$ and $r'$ follow from part (1) of  Proposition~\ref{prop:subsetsofs1odd}.

If $r=\frac{l}{2l+1}$ then the value of $\wfcleq{S^1}{\frac{l}{2l+1}}=\frac{l}{2l+1}$ is attained by the vertex set of any regular $(2l+1)$-gon. Proposition~\ref{prop:vrc-even} implies that $\vrcleq{S^1}{\frac{l}{2l+1}}$ is homotopy equivalent to a wedge of copies of $S^{2l}$, and so it remains to count the number of wedge summands. For $t\in(0,\frac{1}{2l+1})_{S^1}$ let $Y_t=\{\frac{i}{2l+1}, t+\frac{i}{2l+1}~:~i=0,\ldots,2l\}$. We have an isomorphism $\vrleqdir{Y_t}{\frac{l}{2l+1}}=\cnkdir{2(2l+1)}{2l}$, hence each inclusion $j_t: \vrcleq{Y_t}{\frac{l}{2l+1}}\incl \vrcleq{S^1}{\frac{l}{2l+1}}$ determines a homology class $\alpha_t={j_t}_*(\iota_{2l})$ in $\vrcleq{S^1}{\frac{l}{2l+1}}$. Each simplex $\beta_t=[t+\frac{i}{2l+1}~:~i=0,\ldots,2l]$ is a maximal face of  $\vrcleq{S^1}{\frac{l}{2l+1}}$, which appears in the support of $\alpha_t$ but not in any other $\alpha_s$ for $s\neq t$. This implies that the classes $\alpha_t$ are independent, and hence $\redhom_{2l}(\vrcleq{S^1}{\frac{l}{2l+1}})$ contains a free abelian group of rank $\mathfrak{c}$. We get a corresponding upper bound by noting that the cardinality of the set of $2l$-simplices in $\vrcleq{S^1}{\frac{l}{2l+1}}$ is also $\mathfrak{c}$, and hence the cardinality of the wedge sum is $\mathfrak{c}$. 
\end{proof}

\begin{remark}
\label{rem:planar}
Chambers et al.\ \cite[Section 6.(1)]{Chambers2010} asked if for all $k\geq 2$ and any finite subset $X\subseteq \RR^2$ the homology group $\redhom_k(\vrc{X}{r})$ is generated by induced $k$-dimensional cross--polytopal spheres (for all $k$ these are complexes of the form $\cl(\cnk{2k+2}{k})$, where we considered $k=2l$ in this section). Proposition~\ref{prop:crosspoly-classify} confirms this when $k=2l$ for subsets $X\subseteq S^1\subseteq \RR^2$.  When $k=2l+1$ is odd the claim fails already for $X\subseteq S^1$. For example, one can check that for $\frac13<\frac{k}{n}<\frac{3}{8}$ the graph $\cnk{n}{k}$ does not contain an induced subgraph isomorphic to $\cnk{8}{3}$, yet $\redhom_3(\cl(\cnk{n}{k}))=\redhom_3(S^3)\neq 0$ by Theorem~\ref{thm:cnk}.
\end{remark}
\section{\v{C}ech complexes}
\label{sect:cech}

The \v{C}ech complex is another simplicial complex commonly associated with a metric space. For a point $x$ in a metric space $M$, let $\ballless(x;r)$ and $\ballleq(x;r)$ denote the open and closed balls in $M$ with center $x$ and radius $r$.
\begin{definition}
\label{def:cech}
For a subset $X\subset M$ of an ambient metric space $M$ and $r>0$, the \emph{\v{C}ech complex} $\cechless(X,M;r)$ \emph{(}resp.\ $\cechleq(X,M;r)$\emph{)} is the simplicial complex with vertex set $X$, where a finite subset $\sigma\subseteq X$ is a face if and only if $\bigcap_{x\in \sigma}B_<(x;r)\neq\emptyset$ \emph{(}resp.\ $\bigcap_{x\in \sigma}B_\leq(x;r)\neq\emptyset$\emph{)}.
\end{definition}
As before, we will omit the subscript in statements which apply to both $<$ and $\leq$. An equivalent definition of $\cech(X,M;r)$ is as the nerve of the family of balls $\{B(x;r)~:~x\in X\}$. Chazal, de~Silva, and Oudot \cite{ChazalDeSilvaOudot2013} refer to these complexes as \emph{ambient \v{C}ech complexes} with landmark set $X$ and witness set $M$. We have the inclusion $\cech(X,M;r/2)\subseteq \vrc{X}{r}$, and if $M$ is a geodesic space then $\vr{X}{r}$ is the $1$-skeleton not only of $\vrc{X}{r}$ but also of $\cech(X,M;r/2)$. 
\begin{notation}
If $X\subseteq S^1$ then we write $\cech(X;r)$ for $\cech(X,S^1;r)$.
\end{notation} 
If $M=S^1$ then the balls are open or closed arcs, and one can see that finite $\sigma\subset X$ is a face of $\cech(X;r)$ if and only if $\sigma$ is contained in some arc of length $2r$.

One can develop a parallel theory of dismantling, winding fractions, and homotopy types for the complexes $\cech(X;r)$ with $X\subseteq S^1$, leading to straightforward analogues of all the results from this paper. We note that the sequence of critical values $(0,\frac{1}{3},\frac{2}{5},\frac{3}{7},\ldots,\frac{l}{2l+1},\ldots)$ determining the transitions of homotopy types will be replaced in the case of \v{C}ech complexes with the sequence
$(0,\frac{1}{4},\frac{2}{6},\frac{3}{8},\ldots,\frac{l}{2(l+1)},\ldots)$, and we refer the reader to \cite{AAFPP-J} for some results in the case of finite $X$. Instead of pursuing the parallel theory of winding fractions for \v{C}ech complexes, we provide a direct transformation from Vietoris--Rips complexes to \v{C}ech complexes. This transformation recovers most but not all of the results that could be obtained with the parallel theory, and we believe it is of independent interest. To our knowledge, Theorems~\ref{thm:cechleqfull} and \ref{thm:cechlessfull} are the first computation for a non-contractible connected manifold $M$ of the homotopy types of $\cech(M,M;r)$ for arbitrary $r$.

Let $\nicep(S^1)$ denote the power set of $S^1$. If $X\subseteq S^1$ and $a,b\in\RR$ then we write $aX+b=\{(ax+b) \md{1}~:~ x\in X\}$, where it is understood that each point $x$ is represented by a real number in $[0,1)$. 

\begin{theorem}
\label{thm:cech-vr-any}
For each $0<r<\frac12$ let $T_r:\mathcal{P}(S^1)\to\mathcal{P}(S^1)$ be given by
$$T_r(X) = \frac{1}{1+2r}X\ \cup \ \big(\frac{1}{1+2r}\cdot(X\cap[0,2r)_{S^1})+\frac{1}{1+2r}\big).$$
Then the (non-continuous) map $\pi_r:S^1\to S^1$ defined by
$$\pi_r(y)=(1+2r)y \md{1}\quad \mathrm{for}\ y\in[0,1)$$
induces a simplicial homotopy equivalence
$$\pi_r: \vrcleq{T_r(X)}{\frac{2r}{1+2r}}\xrightarrow{\htpyequiv} \cechleq(X;r).$$
\end{theorem}
\begin{figure}
\begin{tabular}{cc}
\includegraphics[scale=0.8]{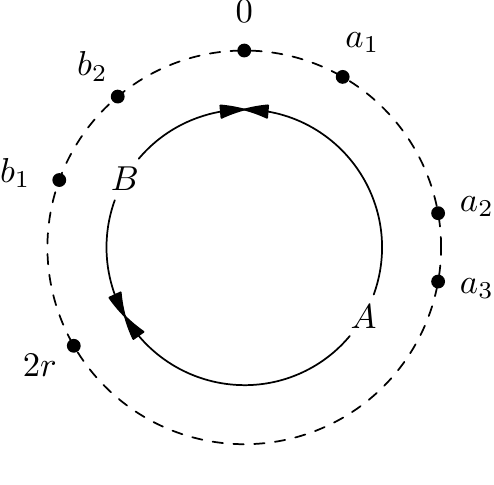} & \includegraphics[scale=0.8]{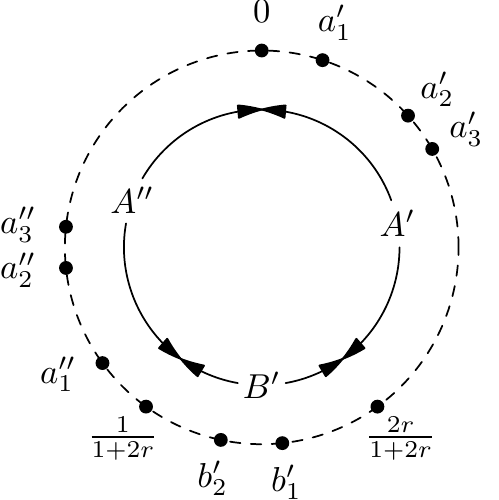}
\end{tabular}
\caption{The action of the operator $T_r$ from Theorem~\ref{thm:cech-vr-any}. \emph{(Left)} A set $X$ split as $X=A\sqcup B$, where $A=X\cap[0,2r)_{S^1}$ and $B=X\cap[2r,1)_{S^1}$. \emph{(Right)} We have $T_r(X)=A'\sqcup B'\sqcup A''$, where $A',A''$ and $B'$ are suitably rescaled and shifted copies of $A$ and $B$. The map $\pi_r:T_r(X)\to X$ sends back $A'$ to $A$, $B'$ to $B$, and $A''$ to $A$.}
\label{fig:tr}
\end{figure}
\begin{proof}
We first verify that $\pi_r(T_r(X))=X$. Take any $y\in T_r(X)$. If $y=\frac{1}{1+2r}x$ for $x\in X$ then $\pi_r(y)=x$. If $y=\frac{1}{1+2r}x+\frac{1}{1+2r}$ for some $x\in X\cap[0,2r)_{S_1}$ then $\pi_r(y)=(x+1)\md{1}=x$. It means that $\pi_r$ restricts to a surjection $\pi_r:T_r(X)\to X$.

Next we check that $\pi_r$ induces a map of simplicial complexes. Let $\sigma$ be any face of the complex $\vrcleq{T_r(X)}{\frac{2r}{1+2r}}$ and let $x_0=\min(\sigma)$, so that $\sigma\cap[0,x_0)_{S^1}=\emptyset$. To prove that $\pi_r(\sigma)$ is a face of $\cechleq(X;r)$ we need to show that it is contained in a closed arc of length $2r$. There are three cases.
\begin{itemize}
\item $x_0\in[\frac{1}{1+2r},1)_{S^1}$. Then $\sigma\subseteq [\frac{1}{1+2r},1)_{S^1}$ and $\pi_r(\sigma)\subseteq [0,2r)_{S^1}$.
\item $x_0\in[\frac{2r}{1+2r},\frac{1}{1+2r})_{S^1}$. Then the only way $x_0$ can be in distance at most $\frac{2r}{1+2r}$ from the other points in $\sigma$ is if that distance is measured clockwise from $x_0$. It means that $\sigma\subseteq [x_0,x_0+\frac{2r}{1+2r}]_{S^1}$ with $x_0+\frac{2r}{1+2r}<1$ and $\pi_r(\sigma)\subseteq[(1+2r)x_0,(1+2r)x_0+2r \md{1}]_{S^1}$.
\item $x_0\in[0,\frac{2r}{1+2r})_{S^1}$. Note that $(x_0-\frac{2r}{1+2r})\md{1}=x_0+\frac{1}{1+2r}$, hence we can write $\sigma\subseteq[x_0,x_0+\frac{2r}{1+2r}]_{S^1}\cup[x_0+\frac{1}{1+2r},1)_{S^1}$. An application of $\pi_r$ gives $$\pi_r(\sigma)\subseteq[(1+2r)x_0,(1+2r)x_0+2r]_{S^1}\cup[(1+2r)x_0,2r)_{S^1}\subseteq
[(1+2r)x_0,(1+2r)x_0+2r]_{S^1}.$$
\end{itemize}
To prove that $\pi_r$ is a homotopy equivalence it suffices to check that the preimage $\pi_r^{-1}(\tau)$ of every face $\tau\in\cechleq(X;r)$ is contractible. The conclusion is then provided by the simplicial version of Quillen's Theorem A due to Barmak \cite[Theorem 4.2]{BarmakTheoremA2011}. Suppose that
$$\tau=\{a_1,\ldots,a_s\}\cup\{b_1,\ldots,b_t\},$$
where possibly $s=0$ or $t=0$, and
$$0\leq a_1<\cdots<a_s<2r\leq b_1<\cdots<b_t<1.$$
The preimage $\pi_r^{-1}(\tau)$ is the subcomplex of $\vrcleq{T_r(X)}{\frac{2r}{1+2r}}$ induced by the vertex set
$$V(\pi_r^{-1}(\tau))=\{a_1',\ldots,a_s'\}\cup\{b_1',\ldots,b_t'\}\cup\{a_1'',\ldots,a_s''\}$$
where $a_i'=\frac{1}{1+2r}a_i$, $b_i'=\frac{1}{1+2r}b_i$, and $a_i''=\frac{1}{1+2r}a_i+\frac{1}{1+2r}$ (see Figure~\ref{fig:tr}).

Note that $\ddist(a_i',b_j')=\frac{1}{1+2r}\ddist(a_i,b_j)$ and $\ddist(b_i',a_j'')=\frac{1}{1+2r}\ddist(b_i,a_j)$ for all $i,j$. Moreover, $\ddist(b_i',b_j')=\frac{1}{1+2r}\ddist(b_i,b_j)$ for $i<j$, $\ddist(a_i',a_j'')=\frac{1}{1+2r}\ddist(a_i,a_j)$ for $i>j$, and $\ddist(a_i'',a_i')=\frac{2r}{1+2r}$.

Since $\tau$ is contained in an arc of length $2r$, there is a vertex $x_0\in\tau$ such that $\tau\subseteq[x_0,x_0+2r]_{S^1}$. We find a vertex $y_0\in V(\pi_r^{-1}(\tau))$ such that $y_0$ is at most $\frac{2r}{1+2r}$ away from every other vertex of $V(\pi_r^{-1}(\tau))$. This will end the proof, since then $\pi_r^{-1}(\tau)$ is a cone with apex $y_0$. We need to consider four cases.
\begin{itemize}
\item If $x_0=a_q$ for some $1\leq q\leq s$ then $y_0=a_q'$. The arc $[a_q,a_q+2r]_{S^1}$ contains all points of $\tau$, therefore $[a_q',a_q'+\frac{2r}{1+2r}]_{S_1}$ contains all points of $V(\pi_r^{-1}(\tau))$ up to the point preceding $a_q''$. Moreover $\ddist(a_q'',a_q')=\frac{2r}{1+2r}$ and $[a_q'',a_q']_{S^1}$ covers the remaining points. 
\item If $x_0=b_1$ and $s=0$ then take $y_0=b_1'$. Since $\ddist(b_1,b_t)\leq 2r$, we have $\ddist(b_1',b_t')\leq\frac{2r}{1+2r}$.
\item If $x_0=b_1$ and $s>0$ then take $y_0=a_s''$. From $\ddist(b_1,a_s)\leq 2r$ we get $\ddist(b_1',a_s'')\leq\frac{2r}{1+2r}$ and in addition $\ddist(a_s'',a_s)=\frac{2r}{1+2r}$. This covers the distances from $a_s''$ to all points of $V(\pi_r^{-1}(\tau))$.
\item The last case, $x_0=b_q$ with $q\geq 2$, is impossible since $\ddist(b_q,b_{q-1})>2r$.
\end{itemize}
\end{proof}
\begin{remark}
Theorem~\ref{thm:cech-vr-any} has no variant for $\mathbf{VR}_<$ and $\cech_<$, since the set $T_r(X)$ contains pairs of points in distance exactly $\frac{2r}{1+2r}$ whose existence is essential for the proof.
\end{remark}

In many natural circumstances maps of \v{C}ech complexes can be lifted to maps of Vietoris--Rips complexes via $\pi_r$. Below we describe the case of inclusions.

\begin{proposition}
\label{prop:cech-inclusion}
Suppose $X\subseteq S^1$ and $0<r\leq r'<\frac12$. Then the (non-continuous) map $\eta:S^1\to S^1$ given by
$$\eta(y)=\frac{1+2r}{1+2r'}\cdot y\quad \mathrm{for}\ y\in[0,1)$$
determines a map of Vietoris--Rips complexes which makes the following diagram commute
$$
\xymatrix{
\vrcleq{T_r(X)}{\frac{2r}{1+2r}} \ar[r]^\eta \ar[d]_{\pi_r}^{\htpyequiv} & \vrcleq{T_{r'}(X)}{\frac{2r'}{1+2r'}} \ar[d]_{\pi_{r'}}^{\htpyequiv}\\
\cechleq(X;r) \ar@{^{(}->}[r]^\subseteq & \cechleq(X;r').
}
$$
\end{proposition}
\begin{proof}
We first verify that $\eta$ gives a well-defined map of Vietoris--Rips complexes in the top row of the diagram; it suffices to check this map on vertices and edges. For vertices, pick any $y\in T_r(X)$. If $y=\frac{1}{1+2r}x$ for $x\in X$ then $\eta(y)=\frac{1}{1+2r'}x\in T_{r'}(X)$. If $y=\frac{1}{1+2r}x+\frac{1}{1+2r}$ for $x\in X\cap[0,2r)$, then $\eta(y)=\frac{1}{1+2r'}x+\frac{1}{1+2r'}$ with $x\in X\cap[0,2r)\subseteq X\cap [0,2r')$, and hence also in this case $\eta(y)\in T_{r'}(X)$. For edges, we suppose that $0\leq y<y'<1$. If $\ddist(y,y')\leq\frac{2r}{1+2r}$ then $\ddist(\eta(y),\eta(y'))\leq \frac{2r}{1+2r}\cdot\frac{1+2r}{1+2r'}\leq\frac{2r'}{1+2r'}$. If $\ddist(y',y)\leq\frac{2r}{1+2r}$ then $\ddist(y,y')\geq\frac{1}{1+2r}$, hence we get $\ddist(\eta(y),\eta(y'))\geq\frac{1}{1+2r}\cdot\frac{1+2r}{1+2r'}=\frac{1}{1+2r'}$ and therefore $\ddist(\eta(y'),\eta(y))\leq\frac{2r'}{1+2r'}$.

Commutativity of the diagram follows from a direct calculation:
$$\pi_{r'}(\eta(y))=(1+2r')\cdot\frac{1+2r}{1+2r'}\cdot y \md{1}=(1+2r)y\md{1}=\pi_r(y).$$
\end{proof}
For arbitrary $X\subseteq S^1$ Theorem~\ref{thm:cech-vr-any} allows one to determine the homotopy type of $\cechleq(X;r)$ from an efficiently constructible instance of the Vietoris--Rips complex. Here are some examples.

\begin{example}
\label{ex:cech-equal}
Let $X_n=\{0,\frac{1}{n},\ldots,\frac{n-1}{n}\}\subseteq S^1$. Then $\cechleq(X_n;\frac{k}{2n})$ is the complex whose maximal faces are generated from $\{0,\frac{1}{n},\ldots,\frac{k}{n}\}$ via rotations by $\frac{1}{n}$. If $r=\frac{k}{2n}$ then $\frac{2r}{1+2r}=\frac{k}{n+k}$ and $T_r(X_n)=X_{n+k}$. We obtain a homotopy equivalence
$$\cl(C_{n+k}^k)=\vrcleq{X_{n+k}}{\tfrac{k}{n+k}}\xrightarrow{\htpyequiv} \cechleq(X_n;\tfrac{k}{2n}).$$
This special case was proved in \cite[Theorem 8.5]{AAFPP-J}.
\end{example}

\begin{theorem}
\label{thm:cechleqfull}
For $0< r<\frac12$ we have a homotopy equivalence
$$
\cechleq(S^1;r)\htpyequiv \begin{cases}
S^{2l+1} & \mathrm{if}\ \frac{l}{2(l+1)}<r<\frac{l+1}{2(l+2)},\ l=0,1,\ldots, \\
\bigvee^{\mathfrak{c}} S^{2l} & \mathrm{if}\ r=\frac{l}{2(l+1)}.
\end{cases}
$$
Moreover, if $\frac{l}{2(l+1)}<r\leq r'<\frac{l+1}{2(l+2)}$ then the inclusion $\cechleq(S^1;r)\incl\cechleq(S^1;r')$ is a homotopy equivalence.
\end{theorem}

\begin{proof}
Note that $r\mapsto\frac{2r}{1+2r}$ is a monotone map which takes the interval $[\frac{l}{2(l+1)},\frac{l+1}{2(l+2)})$ to $[\frac{l}{2l+1},\frac{l+1}{2l+3})$. Since $T_r(S^1)=S^1$ we get a homotopy equivalence
$$\vrcleq{S^1}{\tfrac{2r}{1+2r}}\xrightarrow{\htpyequiv}\cechleq(S^1;r),$$
and the statement of homotopy types now follows from Theorem~\ref{thm:vrcleqfull}.

The statement about inclusions will follow from Proposition~\ref{prop:cech-inclusion} if we show that the map
$\vrcleq{S^1}{\frac{2r}{1+2r}}\xrightarrow{\eta}\vrcleq{S^1}{\frac{2r'}{1+2r'}}$
is a homotopy equivalence. Pick a finite set $Y_0\subseteq S^1$ with $\wfcleq{Y_0}{\frac{2r}{1+2r}}>\frac{l}{2l+1}$. We have a commutative diagram
$$
\xymatrix{
\vrcleq{Y_0}{\frac{2r}{1+2r}} \ar[r]^-\eta \ar@{^{(}->}[d] & \vrcleq{Y_0\cup\eta(Y_0)}{\frac{2r'}{1+2r'}} \ar@{^{(}->}[d]\\
\vrcleq{S^1}{\frac{2r}{1+2r}} \ar[r]^\eta & \vrcleq{S^1}{\frac{2r'}{1+2r'}}.
}
$$
The vertical inclusions are homotopy equivalences by Remark~\ref{rem:finiteapproximation}. The map $\eta$ is injective and preserves the clockwise order of points on $S^1$, hence it is a cyclic homomorphism of the cyclic graphs underlying the top row of the diagram. As $\frac{l}{2l+1}<\wfcleq{Y_0}{\frac{2r}{1+2r}}\leq \wfcleq{Y_0\cup\eta(Y_0)}{\frac{2r'}{1+2r'}}<\frac{l+1}{2l+3}$, the top row is a homotopy equivalence by Proposition~\ref{prop:inducedhtpyequiv}.
\end{proof}

The following is an analogue of Theorem~\ref{thm:vrlessfull} for \v{C}ech complexes in the case when $X=S^1$.

\begin{theorem}
\label{thm:cechlessfull}
For $0<r<\frac12$ we have a homotopy equivalence
$$\cechless(S^1;r)\htpyequiv S^{2l+1} \quad \mathrm{for}\ \tfrac{l}{2(l+1)}<r\leq\tfrac{l+1}{2(l+2)}, \ l=0,1,\ldots.$$
Moreover, if $\frac{l}{2(l+1)}<r\leq r'\leq\frac{l+1}{2(l+2)}$ then the inclusion $\cechless(S^1;r)\incl\cechless(S^1;r')$ is a homotopy equivalence.
\end{theorem}
\begin{proof}
Fix $\frac{l}{2(l+1)}<r\leq \frac{l+1}{2(l+2)}$ and note that $\cechless(S^1,S^1;r)=\colim_n \cechleq(S^1;r-\frac{1}{n})$. All inclusions $\cechleq(S^1;r-\frac{1}{n})\incl \cechleq(S^1;r-\frac{1}{n+1})$ are cofibrations and by Theorem~\ref{thm:cechleqfull} they are self-homotopy equivalences of $S^{2l+1}$ for sufficiently large $n$. That proves the statement of homotopy types.

For the statement about inclusions, note the inclusions $\cechleq(S^1;r-\frac{1}{n})\incl \cechleq(S^1;r'-\frac{1}{n})$ define a natural transformation of diagrams which is a levelwise homotopy equivalence for sufficiently large $n$ by Theorem~\ref{thm:cechleqfull}. It follows that the induced map of (homotopy) colimits $\cechless(S^1;r)\incl\cechless(S^1;r')$ is a homotopy equivalence.
\end{proof}

\section{Concluding remarks}
\label{sect:other}

A natural generalization of our results would be to investigate the complexes $\vrc{M}{r}$ and $\cech(M,M;r)$ for Riemannian manifolds $M$ other than $S^1$, though very little is known along these lines. Intriguing examples include the spheres $S^n$ and tori $(S^1)^n$ for $n\geq 2$. One difficulty is that it is not known whether the homotopy type of $\vrc{M}{r}$ can be approximated by those of complexes $\vrc{X}{r}$ for sufficiently dense subsets $X\subseteq M$. Furthermore, already for $M=S^2$ the complete list of homotopy types of complexes $\vrc{X}{r}$ for finite subsets $X\subset S^2$ is not known.


Towards the goal of understanding Vietoris--Rips complexes of more spaces, we briefly describe two results, the homotopy types of annuli and of tori equipped with the $\ell_\infty$ metric, which can be derived from our computation of $\vrc{S^1}{r}$ using known tools.
\begin{proposition}
\label{prop:derived1}
Consider the annulus $D(\rho,\tilde{\rho})=\{(x,y)\in\RR^2~:~\rho^2\leq x^2+y^2\leq \tilde{\rho}^2\}$ with the Euclidean metric. Then for any $r>0$ the space $\vrcless{D(\rho,\tilde{\rho})}{r}$ is homotopy equivalent to an odd-dimensional sphere or to a point.
\end{proposition}
\begin{proof}
The homotopy which radially deforms the annulus onto its inner boundary does not increase distances, and so it is a \emph{crushing} map in the sense of Hausmann \cite{Hausmann1995}. By \cite[Proposition (2.2)]{Hausmann1995} the inclusion of the Vietoris--Rips complex of $S^1$ into that of $D(\rho,\tilde{\rho})$ is a homotopy equivalence, and so the result follows from Theorem~\ref{thm:vrlessfull}.
\end{proof}

We include a proof of a result for which we were unable to find a published reference.

\begin{proposition}
\label{prop:derived2}
Suppose $(M_1,d_1),\ldots,(M_n,d_n)$ are metric spaces and $M=M_1\times\cdots\times M_n$ is their product equipped with the supremum metric
$$\ell_\infty((x_1,\ldots,x_n),(y_1,\ldots,y_n))=\max\{d_i(x_i,y_i)~|~i=1,\ldots,n\}.$$
Then for any $r>0$ we have a homotopy equivalence
$$\vrc{M}{r}\htpyequiv \vrc{M_1}{r}\times\cdots\times \vrc{M_n}{r}.$$
\end{proposition}
\begin{proof}
For simplicial complexes $K_1, \ldots, K_n$ the categorical product \cite[Definition 4.25]{Kozlov} (in the category of abstract simplicial complexes) is the complex $\prod_i K_i$ with vertex set $V(K_1)\times\cdots\times V(K_n)$ and with faces given by the condition: $\sigma\in\prod_i K_i$ if and only if $\sigma\subseteq \sigma_1\times\cdots\times \sigma_n$ for some $\sigma_i\in K_i$, $i=1,\ldots,n$. Since a subset of $M$ has diameter equal to the maximum of the diameters of its coordinate projections, we get an isomorphism of simplicial complexes
$$\vrc{M}{r}=\textstyle\prod_{i=1}^n\vrc{M_i}{r}.$$
There is a homotopy equivalence $\prod_i K_i\htpyequiv K_1\times\cdots\times K_n$ by \cite[Proposition 15.23]{Kozlov} when each $K_i$ is a finite simplicial complex, and one can see that the finiteness assumption is not necessary by combining the same proof with a version of the nerve lemma for infinite simplicial complexes \cite[Theorem~10.6]{Bjorner1995}. That ends the proof.
\end{proof}

Applied to the torus $\TT^n=(S^1)^n$ the last proposition yields the homotopy types of $\vrc{\TT^n}{r}$ for the $\ell_\infty$ metric on $\TT^n$. It would be interesting to investigate the homotopy types of $\vrc{\TT^n}{r}$ for other $\ell_p$ metrics on $\TT^n$, especially for the $\ell_2$ metric.

%


\section*{Acknowledgements}
We thank Florian Frick, Matthew Kahle, Vin de Silva, and Francis Motta for helpful conversations, and the anonymous referee for helpful remarks concerning the presentation.

\bibliographystyle{plain}
\bibliography{vietoris}

\begin{thebibliography}{10}

\bibitem{Adamaszek2013}
Micha{\l} Adamaszek.
\newblock Clique complexes and graph powers.
\newblock {\em Israel Journal of Mathematics}, 196(1):295--319, 2013.

\bibitem{AAFPP-J}
Micha{\l} Adamaszek, Henry Adams, Florian Frick, Chris Peterson, and Corrine
  Previte-Johnson.
\newblock Nerve complexes of circular arcs.
\newblock {\em Discrete {\&} Computational Geometry}, 56(2):251--273, 2016.

\bibitem{AAM}
Micha{\l} Adamaszek, Henry Adams, and Francis Motta.
\newblock Random cyclic dynamical systems.
\newblock {\em Advances in Applied Mathematics}, 83:1--23, 2017.

\bibitem{BarmakTheoremA2011}
Jonathan~A Barmak.
\newblock On {Q}uillen's {T}heorem {A} for posets.
\newblock {\em Journal of Combinatorial Theory, Series A}, 118(8):2445--2453,
  2011.

\bibitem{Bjorner1995}
Anders Bj{\"o}rner.
\newblock Topological methods.
\newblock {\em Handbook of Combinatorics}, 2:1819--1872, 1995.

\bibitem{Carlsson2009}
Gunnar Carlsson.
\newblock Topology and data.
\newblock {\em Bulletin of the American Mathematical Society}, 46(2):255--308,
  2009.

\bibitem{CarlssonIshkhanovDeSilvaZomorodian2008}
Gunnar Carlsson, Tigran Ishkhanov, Vin de~Silva, and Afra Zomorodian.
\newblock On the local behavior of spaces of natural images.
\newblock {\em International Journal of Computer Vision}, 76:1--12, 2008.

\bibitem{Chambers2010}
Erin~W Chambers, Vin de~Silva, Jeff Erickson, and Robert Ghrist.
\newblock Vietoris--{R}ips complexes of planar point sets.
\newblock {\em Discrete \& Computational Geometry}, 44(1):75--90, 2010.

\bibitem{ChazalDeSilvaOudot2013}
Fr{\'e}d{\'e}ric Chazal, Vin de~Silva, and Steve Oudot.
\newblock Persistence stability for geometric complexes.
\newblock {\em Geometriae Dedicata}, pages 1--22, 2013.

\bibitem{EdelsbrunnerHarer}
Herbert Edelsbrunner and John~L Harer.
\newblock {\em Computational Topology: An Introduction}.
\newblock American Mathematical Society, Providence, 2010.

\bibitem{Gromov1987}
Mikhael Gromov.
\newblock Hyperbolic groups.
\newblock In Stephen~M Gersten, editor, {\em Essays in Group Theory}. Springer,
  1987.

\bibitem{Hatcher}
Allen Hatcher.
\newblock {\em Algebraic Topology}.
\newblock Cambridge University Press, Cambridge, 2002.

\bibitem{Hausmann1995}
Jean-Claude Hausmann.
\newblock On the {V}ietoris--{R}ips complexes and a cohomology theory for
  metric spaces.
\newblock {\em Annals of Mathematics Studies}, 138:175--188, 1995.

\bibitem{HellNesetril2004}
Pavol Hell and Jaroslav Ne\v{s}et\v{r}il.
\newblock {\em Graphs and Homomorphisms}.
\newblock Oxford University Press, 2004.

\bibitem{huntington1916set}
Edward~V Huntington.
\newblock A set of independent postulates for cyclic order.
\newblock {\em Proceedings of the National Academy of Sciences of the United
  States of America}, pages 630--631, 1916.

\bibitem{Imany2008}
Ramin Imany-Nabiyyi.
\newblock The sizes of components in random circle graphs.
\newblock {\em Discussiones Mathematicae Graph Theory}, 28:511--533, 2008.

\bibitem{Kahle2009}
Matthew Kahle.
\newblock Topology of random clique complexes.
\newblock {\em Discrete Mathematics}, 309(6):1658--1671, 2009.

\bibitem{KlamkinNewman1967}
Murray~S Klamkin and Donald~J Newman.
\newblock Extensions of the birthday surprise.
\newblock {\em Journal of Combinatorial Theory}, 3(3):279--282, 1967.

\bibitem{Kozlov}
Dmitry~N Kozlov.
\newblock {\em Combinatorial Algebraic Topology}, volume~21 of {\em Algorithms
  and Computation in Mathematics}.
\newblock Springer, 2008.

\bibitem{Latschev2001}
Janko Latschev.
\newblock Vietoris--{R}ips complexes of metric spaces near a closed
  {R}iemannian manifold.
\newblock {\em Archiv der Mathematik}, 77(6):522--528, 2001.

\bibitem{LinSchwarcfiter2009}
Min~Chih Lin and Jayme~L Szwarcfiter.
\newblock Characterizations and recognition of circular-arc graphs and
  subclasses: {A} survey.
\newblock {\em Discrete Mathematics}, 309(18):5618--5635, 2009.

\bibitem{Matouvsek2008}
Ji{\v{r}}{\'\i} Matou{\v{s}}ek.
\newblock {LC} reductions yield isomorphic simplicial complexes.
\newblock {\em Contributions to Discrete Mathematics}, 3(2), 2008.

\bibitem{Solomon}
Herbert Solomon.
\newblock {\em Geometric Probability}, volume~28 of {\em CBMS-NSF Regional
  Conference Series in Applied Mathematics}.
\newblock SIAM, 1978.

\bibitem{Vietoris27}
Leopold Vietoris.
\newblock {{\"U}ber den h{\"o}heren Zusammenhang kompakter R{\"a}ume und eine
  Klasse von zusammenhangstreuen Abbildungen}.
\newblock {\em Mathematische Annalen}, 97(1):454--472, 1927.

\bibitem{Hocolim}
Volkmar Welker, G{\"u}nter~M Ziegler, and \v{Z}ivaljevi{\'c} Rade~T.
\newblock Homotopy colimits --- comparison lemmas for combinatorial
  applications.
\newblock {\em Journal f{\"u}r die Reine und Angewandte Mathematik (Crelles
  Journal)}, 509:117--149, 1999.

\end{thebibliography}

\appendix
\section{}
\label{sect:appendixA}
We prove the following algebraic fact, which is used in the proof of Proposition~\ref{prop:vrc-even}.

\begin{proposition}
\label{prop:appendixA}
Suppose that $V=\{e_1,\ldots,e_n\}$ is a basis of the free abelian group $\ZZ^n$. Consider the set of $n+{n\choose 2}$ vectors
$$\widetilde{V}=\{e_1,\ldots,e_n\}\cup\{e_i-e_j~:~1\leq i<j\leq n\}.$$
For an arbitrary choice $v,v_1,\ldots,v_k\in \widetilde{V}$, if the subgroup of $\ZZ^n$ generated by $\{v_1,\ldots,v_k\}$ contains some non-zero multiple of $v$, then it also contains $v$.
\end{proposition}
\begin{proof}
We can assume that $v,v_1,\ldots,v_k$ are pairwise distinct. Let $A=\{v,v_1,\ldots,v_k\}$. Note that when expressed in the basis $\{e_1,\ldots,e_n\}$, any two vectors in $\widetilde{V}$ have at most one non-zero coordinate in common. By symmetry it suffices to consider the cases $v=e_1$ and $v=e_1-e_2$.

\smallskip
First suppose $v=e_1$. We have the identity
$$pe_1=\sum_{v_i\in A\setminus\{e_1\}}a_iv_i$$
for some $p,a_1,\ldots,a_k\in\ZZ$ with $p\neq 0$. Consider a labelled graph $G$ with vertex set $A$, where two vectors are connected by an edge with label $i$ ($1\leq i\leq n$) if they both have non-zero $i$-th coordinate. Let $A_1$ be the vertex set of the connected component of $G$ containing $e_1$. Then we still have the identity
$$pe_1=\sum_{v_i\in A_1\setminus\{e_1\}}a_iv_i$$
because the vectors in $A_1$ and $A\setminus A_1$ contribute to two non-overlapping sets of coordinates. 

It is not possible that all the vectors in $A_1\setminus\{e_1\}$ are of the form $e_i-e_j$. Indeed, any linear combination of such vectors has the sum of its coordinates equal to $0$, whereas $pe_1$ does not. Hence the connected component $A_1$ contains some vector $e_l$ with $l\neq 1$. Consider the shortest path in $G$ from $e_l$ to $e_1$. It is easy to see that no edge label appears along this path more than once and that all the intermediate vertices are vectors of the form $e_i-e_j$. The shortest path has the form
$$
e_{l}=e_{l_0} \to \pm(e_{l_0}-e_{l_1}) \to \pm(e_{l_1}-e_{l_2})\to\cdots\to \pm(e_{l_{s-1}}-e_{l_s})\to e_{l_s}=e_1
$$
for some $s\geq 1$, where $l_0=l$ and $l_s=1$, all $l_i$ are pairwise distinct, and $\pm(e_i-e_j)$ stands for $e_{\min(i,j)}-e_{\max(i,j)}$. Now we obtain a presentation
$$
e_1=e_{l_0} +(e_{l_1}-e_{l_0}) + (e_{l_2}-e_{l_1})+\cdots+(e_{l_{s}}-e_{l_{s-1}})
$$
of $e_1$ as a linear combination of elements of $A_1\setminus\{e_1\}$ (with coefficients $\pm1$). That ends the proof of the proposition for $v=e_1$.

\smallskip
The other case, $v=e_1-e_2$, can be reduced to the previous one as follows. Set $e_1'=e_1-e_2$, $e_2'=-e_2$, $e_3'=e_3-e_2$, \ldots, $e_n'=e_n-e_2$. The set $V'=\{e_1',\ldots,e_n'\}$ is a basis of $\ZZ^n$. Moreover, up to signs, the sets $\widetilde{V'}$ and $\widetilde{V}$ coincide. The assumption that $p(e_1-e_2)$ is a combination of $v_1,\ldots,v_k\in \widetilde{V}$ is therefore equivalent to the assumption that $pe_1'$ is a combination of $\pm v_1,\ldots,\pm v_k\in \widetilde{V'}$. From the previous case we get that $e_1'=e_1-e_2$ is also a linear combination of $v_1,\ldots,v_k$. 
\end{proof}

\section{}
\label{sect:appendixB}
In this section we prove equation \eqref{eq:expectR} which gives the expected waiting time for the appearance of an $(\varepsilon,m)$-regular subset in a random sampling of $S^1$.

We will first determine the waiting times for some occupancy problems in the ``balls into bins'' model. Let $K\geq 1$ be the number of bins and fix a constant $m\geq 1$. Consider the following random experiments.
\begin{itemize}
\item[(a)] We throw balls independently and uniformly at random into $K$ bins until one of the bins contains $m$ balls. Let $A_m(K)$ be the random variable denoting the number of balls thrown. By \cite[Theorem 2]{KlamkinNewman1967} we have
$$\expect[A_m(K)]=\Theta(K^{\frac{m-1}{m}})\quad \mathrm{as}\ K\to\infty.$$
This is known as the generalized birthday paradox, the case $m=2$ (and $K=365$ in the folklore formulation) being the classical birthday paradox.
\item[(b)] We throw balls as before, but each time a ball is thrown we assign it, uniformly at random, with one of $m$ colors. When a bin with $m$ balls appears, we call the sequence of colors in that bin, in the order in which they were thrown, the \emph{outcome} of the experiment. The outcome is \emph{good} if all of the $m$ balls have different colors . The number of balls thrown is still given by the random variable $A_m(K)$, since the colors do not influence the stopping condition. Since the balls were colored independently and uniformly, each outcome is equally likely. In particular, the probability of a good outcome is $m!/m^m$.
\item[(c)] We repeat the experiment of (b) until we obtain a good outcome, each time starting with a fresh set of empty bins. Let $\tau$ be the random variable counting the number of repetitions and let $B_m(K)$ be the total number of balls thrown. We have
$$B_m(K)=A_m(K)_1+\cdots+ A_m(K)_\tau,$$
where the $A_m(K)_i$ are independent random variables with the distribution of $A_m(K)$. Clearly $\tau$ is a stopping time with respect to these variables, and by the discussion in (b) we have $\expect[\tau]=m^m/m!$. Now Wald's equation gives
$$\expect[B_m(K)]=\expect[A_m(K)]\cdot\expect[\tau]=
\expect[A_m(K)]\cdot\frac{m^m}{m!}=\Theta(K^{\frac{m-1}{m}})\quad\mathrm{as}\ K\to\infty.$$
\item[(d)] We throw balls independently and uniformly at random into $K$ bins and we color each ball uniformly with one of $m$ colors, until some bin contains at least one ball of each color. If $C_m(K)$ is the random variable counting the number of balls thrown then $A_m(K)\leq C_m(K)\leq B_m(K)$ and  
$$\expect[C_m(K)]=\Theta(K^\frac{m-1}{m})\quad \mathrm{as}\ K\to\infty.$$
In the classical case $m=2$ this is known as the birthday paradox with two types (a boy sharing a birthday with a girl); we were unable to find a literature reference for this result with arbitrary $m$.
\end{itemize}

Recall that $R_m(\varepsilon)$ is the number of points chosen uniformly at random from $S^1$ until an $(\varepsilon,m)$-regular subset appears. We claim that for any integer $K\geq 1$ we have
\begin{equation}
\label{eq:randomineq}
R_m(\tfrac{1}{Km})\leq C_m(K).
\end{equation}
To see this divide $S^1$ into arcs of length $\frac{1}{Km}$. Each union of $m$ arcs whose centers form a regular $m$-gon represents one of our $K$ bins. A uniformly random point $x\in S^1$ can be chosen by picking a uniformly random point $y\in[0,\frac{1}{m})_{S^1}$ and a random number $i\in\{0,\ldots,m-1\}$ and setting $x=y+\frac{i}{m}$; note $y$ determines the bin and $i$ determines the color of the ball. When some bin contains a ball of each color, the corresponding points are within $\frac{1}{Km}$ (in fact even $\frac{1}{2Km}$) from the vertices of a regular $m$-gon.

Next, we claim that
\begin{equation}
\label{eq:randomineq2}
\expect[A_m(K)]\leq 2 \expect[R_m(\tfrac{1}{4Km})].
\end{equation}
Let $P_1$ be the collection of arcs and bins as above, and let $P_2$ be the same collection rotated by $\frac{1}{2Km}$. If $Y$ is $(\frac{1}{4Km},m)$-regular then all points of $Y$ belong to the same bin with respect to $P_1$ or with respect to $P_2$ (or both). Let $p_1$ (resp.\ $p_2$) be the probability that the first time a $(\frac{1}{4Km})$-regular set emerges in the random process $(\nicex_1,\nicex_2,\ldots)$, it is contained in one bin with respect to $P_1$ (resp.\ $P_2$). By symmetry $p_1=p_2$, therefore $p_1\geq\frac12$. If we repeat the whole process until it ends with a set in $P_1$ then the expected number of repetitions is $\expect[\tau]=1/p_1\leq 2$. An argument similar to that in (c) above proves \eqref{eq:randomineq2}.

Letting $K\to\infty$ in \eqref{eq:randomineq} and \eqref{eq:randomineq2} and using the asymptotics of $\expect[A_m(K)]$ and $\expect[C_m(K)]$ we obtain 
$$\expect[R_m(\varepsilon)]=\Theta((\tfrac{1}{\varepsilon})^{\frac{m-1}{m}})\quad \mathrm{as}\ \varepsilon\to 0,$$
which proves \eqref{eq:expectR}.
\end{document}